\documentclass[a4paper,12pt]{amsart}
\usepackage{graphicx, amsmath,amsthm,enumerate,xypic,amssymb,mathrsfs,epsfig}

\usepackage[latin1]{inputenc}
\newtheorem{thm}{Theorem}
\newtheorem*{thmA}{Theorem A}
\newtheorem{cor}{Corollary} 
\newtheorem{lem}{Lemma} 
\newtheorem{prop}{Proposition}
\theoremstyle{definition}
\theoremstyle{remark}
\newtheorem{rem}{Remark}
\numberwithin{equation}{section}

\newcommand{\set}[1]{\left\{#1\right\}}
\newcommand{\R}{\mathbb R}

\newcommand{\Z}{\mathbb{Z}}

\newcommand{\A}{\mathcal{A}}
\newcommand{\C}{\mathcal{C}}
\newcommand{\F}{\mathcal{F}}

\newcommand{\Rd}{\R^d}

\title{Differentiability of functions in the Zygmund class}

\author{Juan Jes{\'u}s Donaire, Jos{\'e} G.\ Llorente and Artur Nicolau}

\thanks{The authors are supported in part by the grants MTM2008-05561, 2009SGR1303, MTM2008-00145 and 2009SGR420}

\begin{document}

\maketitle
\section{Introduction}

Given two parameters $b>1$ and $0<\alpha\leq 1$, in 1876,
 Weierstrass considered the functions given by the series
\begin{equation*}
\sum_{n=0}^\infty b^{-n\alpha} \cos (2\pi b^n x), \qquad x\in \R,
\end{equation*}
and proved that they are continuous and nowhere differentiable if
the parameters satisfy $b^{1-\alpha} \geq 1+3\pi/2$. After a
series of contributions by Dini, Bromwich, Hadamard and others, in
1916, Hardy gave a proof under the optimal assumptions $b>1$ and
$0<\alpha \leq 1$. See \cite{Ha}. The critical case $\alpha =1$ is
the hardest one and it corresponds to the function
\begin{equation}\label{eqq1}
f_b(x)=\sum_{n=0}^\infty b^{-n} \cos (2\pi b^n x), \qquad x\in \R,
\end{equation}
which is in the Zygmund class. The Zygmund class $\Lambda_* (\R^d)$
is the space of bounded continuous functions $f\colon \R^d \to \R$
for which
\begin{equation*}
\| f \|_* =\sup  \{ \frac{|f(x+h)+f(x-h) -2f(x)|}{\|h\|} : x, h \in
\R^d \} < \infty.
\end{equation*}
The small Zygmund class $\lambda_* (\R^d)$ is the subspace formed
by those functions $f\in \Lambda_* (\R^d)$ which satisfy
\begin{equation*}
\lim_{\|h\|\to 0} \sup_{x\in \R^d}  \frac{|f(x+h)+f(x-h)
-2f(x)|}{\|h\|}=0.
\end{equation*}
These spaces were introduced by Zygmund in the forties when he
observed that the conjugate function of a Lipschitz function in the
unit circle does not need to be Lipschitz but it is in the Zygmund
class~\cite{Z1}. For $0<\alpha \leq 1$, let $\Lambda_\alpha (\Rd)$
be the Holder class of bounded functions $f\colon \R^d \to \R$ for
which there exists a constant $C= C(f)$ such that $|f(x+h)-f(x)| \le
C \|h \|^\alpha$, for any $x, h \in \Rd$. It is well known that
$\Lambda_1(\Rd) \subset \Lambda_* (\Rd) \subset \Lambda_\alpha
(\Rd)$ for any $0<\alpha <1$ and actually the Zygmund class
$\Lambda_*(\Rd)$ is the natural substitute of $\Lambda_1(\Rd)$ in
many different contexts. For instance, the Hilbert transform of a
compactly supported function in $\Lambda_1(\Rd)$ may not be in
$\Lambda_1(\Rd)$, while standard Calder{\'o}n-Zygmund operators map
compactly supported functions in $\Lambda_*(\Rd)$ ( respectively in
$\Lambda_\alpha(\Rd)$ for some fixed  $0<\alpha <1$ ) into
$\Lambda_*(\Rd)$ ( respectively into $\Lambda_\alpha(\Rd)$ ). See
\cite{FJW}. The Zygmund class can also be described in terms of
harmonic extensions, Bessel potentials or best polynomial
approximation and again it is the natural substitute of the
Lipschitz class $\Lambda_1 (\Rd)$ in these contexts. See \cite{Z1}
and Chapter~5 of \cite{S}.

A classical result of Rademacher says that any function in
$\Lambda_1 (\Rd)$ is differentiable at almost every point. However
functions in the Zygmund class as the Hardy-Weierstrass function
$f_b$ given in \eqref{eqq1}, may not be differentiable at any point.
More generally, let $g$ be an almost periodic function of class
$\C^2$ in the real line. Then for any $b>1$, the function
\begin{equation*}
f(x)=\sum_{n=0}^\infty b^{-n} g (b^n x), \qquad x\in \R,
\end{equation*}
is in the Zygmund class $\Lambda_* (\R)$ and under mild assumptions
on the function $g$,  Heurteaux has proved that $f$ is nowhere
differentiable \cite{He}. It is worth mentioning that if we allow
having an infinite derivative, then every Zygmund class function
defined in the real line has derivative on a set of points of
Hausdorff dimension 1. See \cite[p.~237]{P} .

In the sixties, Stein and Zygmund proved a series of nice results
relating differentiability properties of functions with the size of
certain natural square functions. Let $f$ be a measurable function
defined in an open set $\Omega \subset \Rd$. Then the set of points
of $\Omega$ where $f$ is differentiable coincides, except at most
for a set of Lebesgue measure zero, with the set of points $x\in
\Omega$ for which there exists $\delta =\delta (x)>0$ such that the
following two conditions hold
\begin{equation*}
\begin{split}
& \sup_{\|h\| <\delta} \frac{|f(x+h)+f(x-h)-2f(x)|}{\|h\|}<\infty,
\\*[5pt] & \int_{\|h\| <\delta}
\frac{|f(x+h)+f(x-h)-2f(x)|^2}{\|h\|^{d+2}}\, dm (h)<\infty.
\end{split}
\end{equation*}
Here $dm$ denotes Lebesgue measure in $\Rd$. Hence,
differentiability can be \linebreak described, modulo sets of
Lebesgue measure zero, by the size of a quadratic \linebreak
expression involving second differences. See \cite[p.~262]{S}.
However these nice results do not apply in the situation we will
consider in this paper where we study differentiability properties
of functions in sets which may have Lebesgue measure zero.

Let $f_b$ be the Weierstrass function given by \eqref{eqq1}. As
mentioned before $f_b \in \Lambda_* (\R)$ and one can see that
\begin{equation*}
\limsup_{h\to 0} \frac{|f_b(x+h)-f_b(x)|}{|h|} = \infty
\end{equation*}
at almost every point $x\in \R$. Actually a Law of the Iterated
Logarithm governs the growth of the divided differences of $f_b$
(see \cite{AP}). Similarly there exist functions in the Small
Zygmund class which are differentiable at almost no point. However
it was already observed by Zygmund in \cite{Z1} that any function in
$\lambda_* (\R)$ is differentiable at a dense set of points of the
real line. Similarly a function in the Zygmund class $\Lambda_*
(\R)$ has bounded divided differences at a dense set of points. In
the eighties, N.\ Makarov proved that Zygmund functions on the real
line have bounded divided differences at sets of Hausdorff dimension
one. See \cite{Ma2} and also \cite{AAP}, \cite{R1} and \cite{R2} for
related results.

\begin{thmA}[Makarov]\

(a) Let $f\in \Lambda_* (\R)$. Then the set
\begin{equation*}
\left \{ x\in \R \, \colon \, \limsup_{h\to 0}
\frac{|f(x+h)-f(x)|}{|h|}<\infty\right \}
\end{equation*}
has Hausdorff dimension $1$.

(b) Let $f\in \lambda_* (\R)$. Then the function $f$ is
differentiable at a set of points of Hausdorff dimension $1$.
\end{thmA}

The main purpose of this paper is to study the situation in higher
dimensions. Given a function $f\in \Lambda_* (\R^d)$ and a unit
vector $e \in \Rd$, let $E(f, e)$ be the set of points where the
divided differences of $f$ in the direction of $e$ are bounded, that~
is,
\begin{equation*}
E(f,e)=\left \{ x\in \Rd \, \colon \, \limsup_{\R \,\ni\, t\to 0}
\frac{|f(x+te)-f(x)|}{|t|}<\infty\right \}.
\end{equation*}
There exist functions $f\in \Lambda_* (\R^d)$ such that, for any
unit vector $e \in \R^d $, the set $E(f,e)$ has Lebesgue measure
zero. However the one dimensional result of  Makarov gives that for
any function $f\in \Lambda_* (\R^d)$ and any fixed unit vector $e
\in \Rd$, the set $E(f,e)$ has Hausdorff dimension $d$. See also
\cite{Ll} and \cite{N}. Similarly for a function $f\in \lambda_*
(\R^d)$ and a unit vector $e\in \Rd$, the set
\begin{equation*}
\left \{ x\in \Rd \, \colon \, \lim_{\R \,\ni\, t\to 0}
\frac{f(x+te)-f(x)}{t}  \textrm{ exists }\right \}
\end{equation*}
may have Lebesgue measure zero but it has Hausdorff dimension $d$.
For a \linebreak fixed direction $e$, the divided differences in
this direction, $ (f(x+te)-f(x)) /t,\break x \in \Rd, $ satisfy a
certain mean value property with respect to Lebesgue measure in
$\Rd$. This is the main point in the proof of Makarov's result as
well as in the arguments leading to the fact that $\dim E(f,e)=d$.
In this paper we want to study the size of the set $E(f)$ of
points where the divided differences in any direction are
simultaneously bounded, that is,
\begin{equation*}
E(f)=\left \{ x\in \Rd \, \colon \, \limsup_{\|h\|\to 0}
\frac{|f(x+h)-f(x)|}{\|h\|}<\infty\right \}.
\end{equation*}
Let $\{e_i \, \colon \, i=1, \dots,d\}$ be the canonical basis of
$\Rd$. If $f\in \Lambda_*(\Rd)$ it turns out that
\begin{equation*}
E(f)=\bigcap_{i=1}^d E (f, e_i).
\end{equation*}
So, the main difficulty in the higher dimensional situation is to
obtain a simultaneous control of the divided differences in
different directions $e_i$, $ i=1,\dots,d$. The first main result
of this paper is the following.

\begin{thm} \label{teo1}\

(a) Let $f$ be a function in $\Lambda_* (\Rd)$. Then the set
$E(f)$ has Hausdorff dimension bigger or equal to $1$.

(b) Let $f$ be a function in $\lambda_* (\Rd)$. Then $f$ is
differentiable at a set of points of Hausdorff dimension bigger or
equal to $1$.
\end{thm}

The result is local in the sense that given $f\in \Lambda_*(\Rd)$
and a cube $Q\subset \Rd$ the set $E(f)\cap Q$ has Hausdorff
dimension bigger or equal $1$. Similarly given $f\in \lambda_*
(\Rd)$ and a cube $Q\subset \Rd$, the function $f$ is
differentiable at a set of points in the cube $Q$ which has
Hausdorff dimension bigger or equal to $1$.

The proof of this result consists on constructing a Cantor type
set on which the function $f$ has bounded divided differences. The
construction of the Cantor type set uses an stopping time argument
based on a certain one dimensional mean value property that the
divided differences of $f$ satisfy. Roughly speaking, the divided
differences distribute their values in a certain uniform way when
measured with respect to length. This is the main new idea in the
proof and it allows us to obtain a simultaneous control of the
behavior of the divided differences in the coordinate directions.
Moreover the result is sharp in the following sense.

\begin{thm}\label{teo2}
There exists a function $f$ in the small Zygmund class
$\lambda_*(\Rd)$ such that the set $E(f)$ has Hausdorff dimension
$1$.
\end{thm}

The one dimensional case may suggest that a natural candidate for
the function $f$ in Theorem~\ref{teo2}
is a lacunary series. However this is not the case. Actually it
turns out that natural lacunary series $f$ in $\Lambda_* (\Rd)$
satisfy $\dim E (f)=~d$  (\cite{DLlN}). Instead, the function $f$
will be constructed as $f=\sum g_k$, where $\{g_k\}$ will be a
sequence of smooth functions defined recursively with $\sum \|g_k
\|_\infty < \infty$. The main idea is to construct them
 in such a way that
$\nabla g_{k+1} (x)$ is \emph{almost orthogonal} to $\nabla
\sum_{j=1}^k g_j (x)$ and $\sum \| \nabla g_{k} (x) \|^2 = \infty$
for \emph{most} points $x\in \Rd$. Since one can not hope to achieve
both requirements at all points $x\in\Rd$, an exceptional set $A$
appears. It turns out that the function $f$ is in the Small Zygmund
class and it is not differentiable at any point in $\R^d \setminus
A$. The construction provides the convenient  \emph{one dimensional}
estimates of the size of the set $A$.

The paper is organized as follows. Next section contains a result on
Hausdorff dimension of certain Cantor sets which will be used in the
proof of Theorem 1. The third section is devoted to the proof of
Theorem~\ref{teo1}. Theorem~\ref{teo2} is proved in Section 4 while
last section is devoted to martingale versions of our results.

The letter $C$ will denote a constant, whose value may change from
line to line, only depending on the dimension $d$. Similarly $C(N)$
denotes a constant depending on the parameter $N$ and the dimension.

It is a pleasure to thank David Preiss for an illuminating idea
concerning Theorem~\ref{teo2}.

\section{Hausdorff dimension of Cantor type sets}

For $n=1,2,\dots$, let ${\mathcal{D}}_n$  be the family of pairwise
disjoint dyadic cubes of the form
$$\left[\frac{k_1}{2^n},\frac{k_1+1}{2^n}\right)\times\cdots\times
\left[\frac{k_d}{2^n},\frac{k_d+1}{2^n}\right),$$ where
$k_1,\ldots,k_d$ are integers. For  $x\in\R^d$, let $Q_n(x)$ be the
unique dyadic cube of ${\mathcal{D}}_n$ which contains the point
$x$.

Given a number $0 \le \alpha \leq d$, the Hausdorff
$\alpha$-content of a set $E \subset \Rd$ is defined as
$$
M_{\alpha}(E)=\inf\biggl\{\sum_j \ell(Q_j)^{\alpha} \biggr\}
$$
where the infimum is taken over all coverings of $E$ by cubes
$\{Q_j\}$ and $\ell(Q_j)$ denotes the sidelength of $Q_j$. The
Hausdorff dimension of $E$ is defined as $\dim
E=\inf\set{\alpha\geqslant 0:M_\alpha(E)=0} $. Next result will be
used to compute the Hausdorff dimension of certain sets appearing
in Theorem ~\ref{teo1}. We will only use the case $s=1$. In
dimension $d=1$ the result can be found in \cite{Hu} or
\cite{Ma1}.

\begin{lem}\label{lema1}
Fix $0< s \leq d$. For $n=1,2,\dots$, let $\A(n)$ be a collection of
closed dyadic cubes in $\Rd$ with pairwise disjoint interiors.
Assume the families are nested, that is,
\begin{equation*}
\bigcup_{Q\in \A(n+1)} Q \subseteq \bigcup_{Q\in \A(n)} Q.
\end{equation*}
Assume also that there exist two constants $0<\eta_0 < K_0 <1$,
satisfying
\begin{enumerate}[(a)]
\item For any $Q\in \A (n+1)$ with $Q\subset \widetilde Q$ for a
certain  $\widetilde Q \in \A(n)$, one has $\ell (Q)\leq \eta_0
\,\ell(\widetilde Q)$. \item For any $\widetilde Q \in \A(n)$ one
has
\begin{equation*}
\sum \ell (Q)^s\geq K_0 \ell (\widetilde{Q})^s,
\end{equation*}
where the sum is taken over all $Q\in \A (n+1)$ contained in
$\widetilde Q$.

\item Furthermore, assume that there exists a constant $\widetilde
K>0$ such that for any cube $Q \subset \Rd$ and any $n=1,2,\dots$
one has
$$\sum \ell (\widetilde Q)^s \leq \widetilde K \ell (Q)^s \, , $$
where the sum is taken over all cubes $\widetilde Q$ in $\A(n)$
contained in $Q$.

Then
\begin{equation*}
\dim \left ( \bigcap_{n=1}^\infty \bigcup_{Q\in \A(n)} Q\right )
\geq s-\frac{\log K_0}{\log \eta_0}\,.
\end{equation*}
\end{enumerate}
\end{lem}

It is worth mentioning that although the constant $\widetilde K$
does not appear in the conclusions, in the case $0 < s<d$ one
cannot achieve any non trivial lower bound for the dimension
assuming only (a) and (b). 

\begin{proof}[Proof of Lemma~\ref{lema1}]
The proof is standard and proceeds by constructing a probability
measure  $\mu\!\geq\! 0$ supported in $E=\bigcap_n \bigcup_{Q\in
\A(n)} Q$ for which there exists a constant $C>0$ such that
\begin{equation}\label{eqq21}
\mu (Q)\leq C\ell (Q)^\alpha
\end{equation}
for any dyadic cube $Q$, where $\alpha\!=\!s-\log K_0/\log \eta_0$.
Then it follows that $M_\alpha (E) \!> \!0$ and finishes the
proof.

The measure $\mu$ will be obtained as a weak limit of certain
measures $\mu_n$ which will be defined recursively as follows. Let
$m_{|E}$ denote the Lebesgue measure in $\Rd$ restricted to the
set $E \subset \Rd$. Without loss of generality we can assume that
$\A(1)$ consists of the unit cube $Q^{(1)}$. We choose
$\mu_1=m_{|Q^{(1)}}$. Fix a positive integer $n$ and assume by
induction that $\mu_{n}$ has been defined. Then consider the
family $\A(n+1)$ and define
\begin{equation*}
\mu_{n+1} =\sum_{Q\in \A(n+1)} \frac{\mu_{n+1} (Q)}{m(Q)}\,m_{|Q},
\end{equation*}
where the masses $\mu_{n+1} (Q)$ are defined as follows. Given $Q\in
\A(n+1)$ let $\widetilde Q$ be the dyadic cube in the previous
family $\A(n)$ which contains $Q$. Also, let $\A(\widetilde Q)$ be~
the~family of dyadic cubes in $\A(n+1)$ contained in $\widetilde Q$.
Then $\mu_{n+1} (Q)$ is defined by the relation
\begin{equation*}
\frac{\mu_{n+1} (Q)}{\ell (Q)^s}=\frac{\mu_{n} (\widetilde
Q)}{\sum_{Q^* \in \A(\widetilde Q)} \ell (Q^*)^s}.
\end{equation*}
Observe that
\begin{equation}\label{eqq22}
\sum_{Q\in \A(\widetilde Q)} \mu_{n+1} (Q) =\mu_{n} (\widetilde Q).
\end{equation}
Hence $\mu_{n+1} (\Rd)\!=\!\mu_{n} (\Rd)$ and iterating we deduce
that $\mu_n$ are probability measures. Let $\mu$ be a weak limit
of $\mu_n$. Then $\mu$ is a probability measure supported in
$\bigcap_n\bigcup_{Q\in \A(n)} Q$. The rest of the proof is
devoted to show the growth condition~\eqref{eqq21}. Observe that
\eqref{eqq22} tells that $\mu (Q)=\mu_n (Q)$ for any cube $Q\in
\A(n)$. Let $Q\in \A(n+1)$. As before we denote by $\widetilde Q$
the cube in the family $\A(n)$ with $Q\subset \widetilde Q$ and by
$\A(\widetilde Q)$ the family of dyadic cubes in $\A(n+1)$
contained in $\widetilde Q$. Then \eqref{eqq22} and property (b)
yield
\begin{equation*}
\frac{\mu (Q)}{\ell (Q)^s}=\frac{\mu_{n+1} (Q)}{\ell (Q)^s}=
\frac{\mu_n (\widetilde Q)}{\sum_{Q^* \in \A(\widetilde Q)} \ell (
Q^*)^s} \leq \frac{1}{K_0}\, \frac{\mu_n (\widetilde Q)}{\ell
(\widetilde Q)^s}\,.
\end{equation*}
Iterating this inequality we obtain
\begin{equation}\label{eqq23}
\frac{\mu (Q)}{\ell (Q)^s} \leq K_0^{-n-1}, \qquad Q\in \A(n+1).
\end{equation}
We now prove the estimate \eqref{eqq21}. Let $R$ be a dyadic cube
in $\Rd$, $\, \ell (R)\leq 1$. Choose a positive integer $n$ such
that $\eta_0^{n+1}\leq \ell (R)\leq \eta_0^n$. Then
\begin{equation*}
\mu (R) \leq \sum \mu (Q),
\end{equation*}
where the sum is taken over cubes $Q$ in the family $\A(n+1)$ with
$Q\cap R \neq \emptyset$. Since the cubes in $\A(n+1)$ have
sidelength smaller than $\eta_0^{n+1}$, if $Q\cap R\neq \emptyset$
we deduce that $Q\subset 3R$, where $3R$ is the cube concentric to
$R$ whose sidelength is $3 \ell (R)$. Hence estimate \eqref{eqq23}
tells
\begin{equation*}
\mu (R) \leq K_0^{-n-1} \sum_{Q\in \A(n+1) \atop Q\subset 3R } \ell
(Q)^s
\end{equation*}
and property (c) gives that
\begin{equation*}
\mu (R) \leq K_0^{-n-1} \widetilde K 3^s \ell (R)^s.
\end{equation*}
Now
\begin{equation*}
K_0^{n+1}=\eta_0^{(n+1)\log K_0 /\log \eta_0} \geq \eta_0^{\log
K_0 /\log \eta_0} \ell (R)^{\log K_0/\log \eta_0}
\end{equation*}
and we deduce that
\begin{equation*}
\mu (R)\leq 3^s \widetilde K \eta_0^{-\log K_0 /\log \eta_0} \ell
(R)^{s-\log K_0/\log \eta_0}
\end{equation*}
which gives \eqref{eqq21} and finishes the proof.
\end{proof}

\section{Theorem~\ref{teo1}}

Let $\{e_1, e_2, \ldots , e_d \}$ be the canonical basis of $\Rd$.
Let $Q$ be a cube in $\R^d$ with edges parallel to the coordinate
axis. We will say that $Q$ has the origin at the point $x\in\R^d$
if its vertexes are the points
$x+ \ell(Q)(\varepsilon_1 e_1+\varepsilon_2 e_2+\ldots+ \varepsilon_d e_d)$,
where $\varepsilon_j\in\{0,1\}$, $j=1,\cdots, d$ and $\ell(Q)$ is
the sidelength of $Q$. Given a cube $Q$ in $\R^d$ with origin $x$
and sidelenght $\ell(Q)$ and given a function $f\colon \R^d\to\R$, let
$V(Q)$ be its discrete gradient at the cube $Q$ given by
$$V(Q)=\sum_{j=1}^d \frac{f( x+ \ell(Q) e_j ) -f(x)}{\ell(Q)}
e_j . $$
If $f$ is a function in the Zygmund class the vector
$V(Q)$ behaves as a sort of gradient as the following
lemma shows.

\begin{lem}\label{lema2}
Let $f\colon \R^d\to\R$ be a function in the Zygmund class $\Lambda_* (\R^d)$ and
let $Q$ be a cube in $\R^d$. Assume that $a$ and $b$ are two
points in $\R^d$ for which there is a constant $C>0$ such that
$dist(a,Q)\leq C\ell(Q)$ and $\|b-a\|\leq C\ell(Q)$. Then there
exists a constant $K=K(C,d)$ only depending on $C$ and the
dimension
 $d$ such that
$$\left|f(b)-f(a)-\langle V(Q), b-a \rangle\right|\leq K\, \|f\|_*\,
\ell(Q).$$
\end{lem}

\emph{Proof.} Without loss of generality we may assume that $Q$ is
the unit cube in $\Rd$ and $2\|b\|  + \|a \| \le 1$. Since the
function $f$ is in the Zygmnund class we have
$$
| f(a) + f(b-a) - 2f(b/2)  | \leq  \|f \|_*
$$
and
$$
| f(b) + f(0) - 2f(b/2)  | \leq  \|f \|_* \, .
$$
Hence
\begin{equation}\label{eqq24}
| f(b) - f(a) - (f(b-a) - f(0))  | \leq 2 \|f \|_* \, .
\end{equation}
So we can also assume that the point $a$ is the origin.

Let us first discuss the particular case when $b$ lies in a
coordinate axis, that is, $b = b_k e_k$ for a certain $\,k=1,2,
\ldots , d$, with $b_k \in \R$, $\,|b_k|\leq 1/2$. Since the
function $f$ lies in the Zygmund class, the divided differences of
$f$ in a fixed direction can grow at most as a fixed multiple of
the logarithm of the increment. Hence there exists a universal
constant $K_1$ such that
$$
\left | \frac{f(b) - f(0)}{b_k} - (f(e_k) - f(0)) \right | \le K_1 \|f\|_* \log
\frac{1}{\|b \|} \, .
$$
Since $|b_k| = \|b \|\le 1/2$ we deduce that
$$
| f(b) - f(0) - b_k (f(e_k) - f(0)) | \le K_1 \|f\|_* \, .
$$
We now discuss the general case. Write $b=\sum_{k=1}^{d} b_k e_k
$, where $|b_k | \le 1/2$, $\, k=1,2,\ldots,d$. Then
$$
f(b)- f(0) = f(b_1 e_1) - f(0)  +  \sum_{k=2}^{d} f(\sum_{j=1}^{k}
b_j e_j ) - f(\sum_{j=1}^{k-1} b_j e_j ) \, .
$$
Applying \eqref{eqq24}, for any $k=2,\ldots, d$ we get
$$
\left |  f\left ( \sum_{j=1}^k b_je_j\right ) - f \left ( \sum_{j=1}^{k-1} b_je_j\right ) -
(f(b_ke_k)-f(0))       \right | \leq 2\|f\|_*
$$
Now each term is in the situation of the particular case discussed
in the previous paragraph and we obtain
\begin{align*}
& | f(b_1 e_1) - f(0) - b_1 (f(e_1) - f(0)) | \le K_1 \|f\|_*  \, ,  \\*[5pt]
&  \left | f \left (\sum_{j=1}^{k} b_j e_j \right ) - f \left (\sum_{j=1}^{k-1} b_j e_j \right ) -
b_k (f( e_k) - f(0)) \right | \le (K_1+2) \|f\|_* \, ,
\end{align*}
for any $ k=2, \ldots, d$. Hence
$$
| f(b) - f(0)   - \sum_{k=1}^{d} b_k (f( e_k) - f(0)) | \le (K_1+2) \,d
\|f\|_*
$$
which finishes the proof. \hfill$\qed$

As a consequence of the Lemma, in the proof of Theorem~\ref{teo1} ,
instead of studying the divided differences of a function in the
Zygmund class, we can restrict attention to the behavior of the
discrete gradient.

\begin{cor}~\label{coro1}\

(a) Let $f$ be a function in the Zygmund class. Let $x \in \Rd$.
Then $x \in E(f)$ if and only if $ \limsup_{n \to \infty}
\|V(Q_n(x)) \| < \infty$, where $V(Q)$ denotes the discrete
gradient of $f$ at the cube $Q$.

(b) Let $f$ be a function in the small Zygmund class. Then $f$ is
differentiable at  a point $x \in \Rd$ if and only if
 $ V(Q_n(x)) $ has limit when $n \to \infty$.
\end{cor}

The proof of Theorem~\ref{teo1} consists on constructing a Cantor
type set on which the function $f$ has bounded divided differences.
The construction of the Cantor type set uses an stopping time
argument based on the following auxiliary result which contains the
main idea of the construction.  The key point is that, when the
divided differences are large, the discrete gradient distributes its
values in a certain uniform way when measured with respect to
length.

If  $f$ is in the Zygmund class $\Lambda_* (\R^d)$ and  $V(Q)$ is
its discrete gradient at the cube $Q$, consider the function
\begin{equation}\label{omega}
w(\delta)= \sup \{|V(Q')-V(Q)| \},
\end{equation}
where the supremum is taken over all pairs of cubes $Q' \subset Q
\subset \Rd$ with $2 \ell(Q') = \ell(Q)  \leq \delta$. It follows
from Lemma~\ref{lema2}  that there exists a constant $C$ only depending on the
dimension such that $w(\delta ) \le C \|f \|_*$, for any $\delta
>0$.

\begin{prop}\label{propo1}
Let $f$ be a function in the Zygmund class $\Lambda_* (\R^d)$
satisfying
\begin{equation}\label{derinfi}
H^1 (\{x\in\R^d\, :\, \limsup_{\|h\| \to
0}\frac{|f(x+h)-f(x)|}{\|h\|}<\infty\}) =0.
\end{equation}
Consider the corresponding function $w(\delta)$  defined in
\eqref{omega} above and assume that $w(\delta) \leq 1$ for any
$\delta >0$. Fix $0< \varepsilon <1$. Then there is a positive
constant $C_0 = C_0 (\varepsilon , d)$ such that for any $\delta >
0$ there is $M_0 = M_0 ( \delta ) > 0$  with the following property.
If $M > M_0 $ and $Q \subset \R^d$ is a dyadic cube with $l(Q) <
\delta $ and $V(Q) = Me$
 where
$e\in \R^d$, $\| e \|= 1 $  then for each unit vector $u\in \R^d$,
there is a finite family $\mathcal{Q^*}= \{Q_j^*  \}$ of dyadic
subcubes of $Q$ such that


\begin{enumerate}[(a)]
\item For any $j$, one has $\ell(Q_j^*)\leq 2^{- \varepsilon
M}\ell(Q)$ and $\varepsilon M \leq\|V(Q_j^*)- M e \|\leq
\varepsilon M + w(\ell(Q))$.

\item If $x\in Q_j^*$   and
$-\log_2\ell(Q)\leq n <-\log_2\ell(Q_j^*)$ for some $j$ and $n$,
then  $\| V(Q_n(x))-M e \|\leq \varepsilon M$.

\item   $ \sum_j\ell(Q_j^*)\geq C_0^{-1} \ell(Q)$, where
the sum is taken over all cubes $Q_j^* \in \mathcal{Q^*}$.
Moreover for any cube $R$ contained in $Q$, one has $
\sum_j\ell(Q_j^*) \leq C_0 \ell(R) $, where the sum is taken over
all cubes $Q_j^*  \in \mathcal{Q^*}$ contained in $R$.

\item For any $j$, one has $\langle V(Q_j^*)-M e, u \rangle\geq
\frac{2}{3}\varepsilon^2 M$.
\end{enumerate}
\end{prop}

Let us first explain the main idea of the proof. Fixed the dyadic
cube $Q$ with $V(Q)=M e$, consider the family $\{Q_j \}$ of
maximal dyadic subcubes $Q_j$ of $Q$ satisfying
$$\| V(Q)- V(Q_j) \|\geq \varepsilon M$$
It easily follows that the whole family $\{Q_j \}$ satisfies
conditions (a) and (b). In particular $ V(Q_j)$ lays, up to a
bounded term, in the sphere centered at $V(Q)$ of radius
$\varepsilon M$. Condition (d) tells that we are only interested in
the subfamily $\mathcal{Q^*}= \{Q_j^* \}$ of $\{Q_j \}$ of cubes
$Q_j^*$ such that the corresponding discrete gradient $V(Q_j^*)$
lies in a certain cone with vertex at $V(Q)$. See
figure~\ref{figura1}. Condition (c) tells that when measured with
respect to length, there is a fixed proportion of cubes $\{Q_j^* \}$
satisfying (d). This is the main point in the result. We will
construct a polygonal path joining two parallel faces of a duplicate
of $Q$ on which the discrete gradient satisfies, up to bounded
terms, a certain mean value property. The family $\mathcal{Q^*}$
will be chosen as a subfamily of the cubes $\{Q_j \}$ which have a
non empty intersection with this polygonal and the mean value
property will lead to estimate (c).

 \begin{figure}
\begin{center}
\epsfig{figure=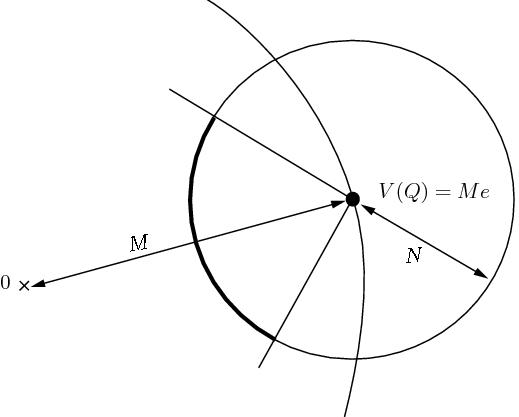}
 \caption{The bold region corresponds to the possible locations of
 $V(Q_j^*)$ when $u=-e$.}\label{figura1}
 \end{center}
\end{figure}

\emph{Proof of Proposition~\ref{propo1}.} Without loss of generality
we may assume that $ \|f \|_* =~1$.  Let $N = \varepsilon M$. The
proof is presented in several steps.

\bigskip

1. \textsl{Reduction to the case $u=-e$.}

Assume we have proved the Proposition in the case $u=-e$ and let
us show it for another unit vector $u$. Consider the function
$g(x)=f(x)-M\langle u+ e,x\rangle$. It is clear
that $g\in\Lambda_* (\R^d)$ and  $\|g\|_*=\|f\|_*$. Moreover, the
discrete gradient $V_g$ of the function $g$ verifies that for any dyadic
cube $R$,
$$V_g(R)=V(R)-M(u+ e).$$
Then $V_g(Q)=-M u$ and, according to our assumption, we can choose
the family $\{Q_j^* \}$ of cubes given by the Proposition, for
which $N\leq\|V_g(Q_j^*)-V_g(Q)\|\leq \break N+ w(\ell(Q))$ and
$\langle V_g(Q_j^*) \!-\! V_g(Q),-\frac{V_g(Q)}{\|V_g(Q) \|}
\rangle \!\geq\! \frac{2}{3}\varepsilon N$. Hence
$N\leq\|V(Q_j^*)- V(Q)\| \break \leq N+ w(\ell(Q))$ and $\langle
V(Q_j^*)-M e, u \rangle\geq \frac{2}{3}\varepsilon N$, for any
$j$.

\bigskip

2. \textsl{Covering the cube $Q$.}

Let $\{ Q_j \}$ be the family of maximal dyadic cubes contained in
$Q$ so that
$$\|V(Q)- V(Q_j)\|\geq N.$$
The family of cubes $\mathcal{Q^*} $ in the statement will be a
subfamily of $\{ Q_j \}$.  The maximality of the cubes $\{ Q_j \}$
imply that for any $j$ we have
\begin{itemize}
\item $ N\leq \|V(Q_j)-M e\|\leq N+ w(\ell(Q))$

\item $\| V(Q_n(x))-M e \|< N$ if $x \in Q_j$ for some $j$ and the
integer $n$ satisfies $-\log_2\ell(Q)\leq n <-\log_2\ell(Q_j)$.
\end{itemize}
Moreover the condition $w(\delta ) \leq 1$ gives that
$\ell(Q_j)\leq 2^{-N}\ell (Q)$ for any $j$. So the whole family
$\{ Q_j \}$ satisfy the first two properties in the statement and
the rest of the proof consists on finding a subfamily verifying
the last two properties. Since condition \eqref{derinfi} holds,
Corollary~\ref{coro1}  tells that
$$H^1  (\{x\in\R^d\, :\, \limsup_{n \to \infty} \|V(Q_n(x))\|<\infty\})=0$$
and hence $H^1(Q\setminus\cup Q_j)=0$. Fix a small constant $\eta
>0$  with  $\eta \le  1/4M \sqrt d$ and find a covering of
$Q\setminus\cup Q_j$ by dyadic cubes $\{ R_j\}$ so that
\begin{equation}\label{eta}
\sum_j\ell(R_j)<\eta \ell(Q).
\end{equation}
We can assume that no $R_\ell$ is contained in $\cup Q_j$.
If $\frac12 Q$ denotes the closed cube, centered at the center of
$Q$, with sidelength $ \ell(Q) /2$, it is clear that $\frac12 Q
\subset  \cup \tilde{Q_j} \cup \tilde{R_j }$, where
$\{\tilde{Q_j}\}$ and $\{\tilde{R_j}\}$  are respectively, the open
cubes centered at the centers of $Q_j$ and $R_j$ with double
sidelength. Compactness of $\frac12 Q$ allows us to obtain a finite
covering that we will denote by $C_1, C_2,\ldots, C_n$, that is,
$\frac12 Q \subset  C_1 \cup \ldots C_n \, .$
 Now, we label each $C_j$ with a unit vector
$v_j$ in the following way:
\begin{itemize}
\item If $C_j = \tilde{Q_k}$ for some $k$, pick
$v_j=\frac{V(Q_k)}{\|V(Q_k)\|}$, \item if $C_j = \tilde{R_k}$ for
some $k$, pick $v_j=e$.
\end{itemize}
Using the cubes $\{C_i : i=1, \ldots , n \}$ and their labels, we
will construct a polygonal path $\Gamma$, starting from a point in
$\frac12 Q$  in such a way that $\text{diam}\, (\Gamma)$ is
comparable to $\ell(Q)$. The polygonal will follow the directions of
the vectors $\{v_j\}$ and will have vertexes at certain points $a_1,
a_2, \ldots , a_n$ so that all of them, except the last one, lie in
$\frac12 Q$

\bigskip

3. \textsl{Choosing the first vertexes  $a_j$.}




Let $a_1$ be a fixed point in the cube $\frac14 Q$.  Next vertexes
are defined recursively. Assuming that $a_j \in \frac{1}{2} Q$ has
been defined, let us describe the choice of $a_{j+1}$. Since
$a_j\in \frac12 Q$, there is a cube $C_k$, $k=k(j)$, containing
this point, and recall that attached to this cube there is a unit
vector $v_k$. Let us define
$$a_{j+1}=a_j+\text{diam}\, (C_k) v_k.$$
Since $C_k$ is an open cube, it is clear that $a_{j+1}\not\in
C_k$. Since $\varepsilon <1$, there exists a constant $\alpha_0 <
\pi /2$ such that the angle $\alpha_k$ between $v_k$ and $e$
satisfies $|\alpha_k | < \alpha_0$ for any $k$. Let $\pi(x)$ be
the scalar product defined by $\pi (x) = \langle x, e\rangle$ ,
$x\in\R^d$. Observe that $\pi ( a_{j+1} - a_j )  >  (\cos \alpha_0
) \text{diam}\, C_k$. Since there are finitely many cubes $C_k$,
iterating this process, in a finite number of steps we find a
vertex not contained in $\frac12 Q$. We stop the process whenever
we obtain a vertex $a_n$ verifying $a_n \notin \frac12 Q$. Observe
that there exists a constant $C_0$ such that $\ell(Q) / 4 < \|a_n
- a_1\| < C_0 \ell(Q)$.

Let $a_1, a_2,\ldots, a_n$ be the points obtained by the previous
algorithm and let us denote by $J'$ the subset of $\{1,2,\ldots
,n\}$ formed by the indexes $j$ for which the selected cube $C_k$
was a cube in the family $\{\tilde{R_m} \}$. Put $J=\{1,2,\ldots
,n\}\setminus J^{'}$. Write
\begin{align*}
f(a_n)-f(a_1)&=\sum_{j=1}^{n-1} (f(a_{j+1})-f(a_j))\\
&=\sum_{j\in J'}(f(a_{j+1})-f(a_j))+ \sum_{j\in
J}(f(a_{j+1})-f(a_j))
\end{align*}
Since $\|a_n - a_1 \| < C_0 \ell(Q)$ we can apply
Lemma~\ref{lema2} to obtain
\begin{equation}\label{ident32}
|f(a_n)-f(a_1) - \langle V(Q), a_n-a_1\rangle  | \le C\ell(Q),
\end{equation}
where $C$ is an absolute constant only depending on the dimension.
Since $ V(Q)=M e$, we have $\langle V(Q), a_n - a_1 \rangle = M (\pi
(a_n) - \pi (a_1) )$ and \eqref{ident32} tells
\begin{equation}\label{ident}
\bigg |\sum_{j\in J'}(f(a_{j+1})-f(a_j))+ \sum_{j\in J}(f(a_{j+1})-f(a_j))
- M (\pi (a_n)-\pi(a_1)) \bigg  | \le C\ell(Q)
\end{equation}

 The family $\mathcal{Q^*}$ will be a subfamily
of $\{Q_j : j \in J\}$ while the first sum in \eqref{ident} will
be an error term.

\bigskip

4. \textsl{Estimating the error term.}

If $j\in J'$  the cube $C_k$, $k=k(j)$, containing $a_{j}$ belongs
to the family $\{ \tilde{R}_m \}$. So, write $C_k = \tilde{R}_{\ell}$.
Since $R_{\ell}$ is not contained in $\cup Q_j$, we have that $\|
V(R_\ell)-M e\|<N$. Hence by Lemma~\ref{lema2}, there exists a fixed
constant C, such that  $|f(a_{j+1}) - f(a_j)| \le (M+N+C)
\|a_{j+1} - a_j \| \le 2M \|a_{j+1} - a_j \|$ if $M$ is
sufficiently large. Since by \eqref{eta} the amount of length
collected by these squares is small, we make the trivial estimate
on them:
\begin{align*}
 \sum_{j \in J^{'}} |f(a_{j+1})-f(a_j)|  &\leq  4 \sqrt d M \sum_{j\in J^{'}} \ell(R_j)\leq
 4 \eta  \sqrt{d} M \ell(Q).
\end{align*}
Since $\eta >0$ was chosen so that $ \eta < 1/4M \sqrt d$, estimate
\eqref{ident} gives that
\begin{equation}\label{ident2}
| \sum_{j\in J}(f(a_{j+1})-f(a_j)) - M (\pi (a_n)-\pi(a_1)) | \le
\tilde{C} \ell(Q),
\end{equation}
where $\tilde{C}$ is a constant only depending on the dimension.

\bigskip

5. \textsl{Choosing the cubes of the family $\mathcal{Q^*}$.}

Assume $j\in J$. This means that  $ a_j \in C_k  \in \{ \tilde{Q}_k
\}$. To simplify the notation reorder the family $\{ \tilde{Q}_k \}$
so that $a_j \in \tilde{Q}_j$ for $j \in J$. Consequently we have
$$a_{j+1}=a_j+  \text{diam}\, (\tilde{Q}_j)  v_j \, .$$
For $j\!=\!1,2, \ldots , n$, consider the segment between $a_j$ and
$a_{j+1}$,  and let $\Gamma$ be the polygonal with vertexes $\{a_1,
\ldots , a_n \}$. Since $\|a_{j+1}-a_j\|$ is comparable to
$\ell(Q_j)$, we can apply Lemma~\ref{lema2} again to get
\begin{equation}\label{eq359}
| f(a_{j+1})-f(a_j)  - \langle V(Q_j),a_{j+1}-a_j\rangle | \le C
\ell (Q_j) \, .
\end{equation}
 Since the angle $\alpha_j$ between the vectors
$v_j$ and $e$ satisfies $| \alpha_j | < \alpha_0 < \pi /2$, we
have that $\|a_{j+1} -a_j \| \leq C_1 \pi (a_{j+1}-a_j)$ for any
$j$, where $C_1 = 1 / \cos \alpha_0$. Hence
\begin{equation}\label{eq350}
\sum \ell (Q_j) \leq C_1 \|a_n-a_1\| \leq C_2 \ell (Q).
\end{equation}
The choice of $a_{j+1}$ gives that for any $j\in J$, we have
$$\langle V(Q_j),a_{j+1}-a_j\rangle=\|V(Q_j)\|\, \|a_{j+1}-a_j\|=\frac{\|V(Q_j)\|^2}{\langle V(Q_j),e\rangle}\, (\pi(a_{j+1})-\pi(a_j)).$$
So estimates \eqref{ident2}, \eqref{eq359}  and \eqref{eq350} give
the following one dimensional mean value property
\begin{equation}\label{bofa} \sum_{j\in J} \frac
{\|V(Q_j)\|^2}{\langle V(Q_j), e\rangle} (\pi(a_{j+1})-\pi(a_j)) =
M (\pi (a_n)-\pi(a_1)) + 0(\ell(Q)) \, .
\end{equation}
Here $0(\ell(Q))$ denotes a quantity which is a bounded
by a fixed proportion, independent of $Q$, of $ \ell(Q)$.
%
Now, a simple calculation gives
\begin{align} \label{eq351}
\frac{\| V(Q_j)\|^2}{\langle  V(Q_j), e \rangle}\,
&= \frac{\| V(Q_j)-M e\|^2+2M\langle V(Q_j)-M e,e\rangle +M^2}{\langle V(Q_j)-M e, e\rangle +M}\\
&= M\left(2-\frac{1-\frac{1}{M^2}\|V(Q_j)-M
e\|^2}{1+\frac{1}{M}\langle V(Q_j)-M e, e\rangle}\right). \nonumber
\end{align}
Decompose the set $J$ into \textsl{good} and \textsl{bad} indexes,
that is, $J=B \cup G$ where
\begin{align*}
G&=\Big  \{ j\in J \,: \, \langle V(Q_j)- M e,e\rangle\leq
-\frac{2}{3}\varepsilon
N  \Big \}, \\
B &= J \setminus G
\end{align*}
The family of cubes $\mathcal{Q^*}$ we are looking for is precisely
 $\{Q_j : j \in G \}$. It is clear that property (d), with $u=-e$, holds.
 Let $R$ be a cube in $\Rd$ and assume that
 $Q_j^* \in \mathcal{Q^*}$ is contained in $R$. Then there exists
 a vertex $a_j$ of the polygonal $\Gamma$ with $ a_j \in 2 Q_j^*
 \subset 2R$. Since $\ell (Q_j^*) \leq \|a_{j+1} - a_j \| \le (\cos^{-1} \alpha_0 ) \pi (a_{j+1} - a_j )$,
  we deduce that there exists a constant $C_3 = C_3 (\alpha_0 , d)$ such that $\sum \ell (Q_j^*) < C_3 \ell (R)$ where the sum is taken over all  cubes $Q_j^*$ contained in $R$. This is the second part of
 statement (c).

\medskip

6. \textsl{Estimates on G and B}.

The stopping time process described in Paragraph 2 gives that
$\|V(Q_j)- M  e\|^2\geq N^2$. Hence $\langle V(Q_j)- M e,
e\rangle\geq -N$ and if $j\in B$ we also have $\langle V(Q_j)- M e,
e \rangle \break \geq - 2 \varepsilon N /3$. Consequently, using
\eqref{eq351} we obtain the following estimates
\begin{itemize}
\item If $j\in G$, one has
$\displaystyle\frac{\|V(Q_j)\|^2}{\langle V(Q_j), e\rangle}\geq M
(1-\varepsilon)$,
\smallskip

\item If $j\in B$, one has
$\displaystyle\frac{\| V(Q_j)\|^2}{\langle V(Q_j), e\rangle}\geq
M\left( 2-\frac{1-\varepsilon^2}{1- 2 \varepsilon ^2 /3 } \right)
= M(1+ C(\varepsilon))$,
\end{itemize}
where $C(\varepsilon ) =  \varepsilon^2 / (3 - 2 \varepsilon^2)
> 0$.
Hence equation (\ref{bofa}) gives
\begin{align}
\pi(a_n)-\pi(a_1)+ \frac{O(\ell (Q))}{M}  &\geq (1-\varepsilon)
\sum_{j\in G}( \pi(a_{j+1})-\pi(a_j) )  \\
&\quad + (1+ C(\varepsilon))\sum_{j\in B} ( \pi(a_{j+1})-\pi(a_j) )
\nonumber
\end{align}
Since $\sum \ell (R_j) < \eta \ell (Q) <\ell (Q)/M$ we have that
$$\sum_{j\in J}
\pi(a_{j+1})-\pi(a_j)=\pi(a_{n})-\pi(a_1) + \frac{O( \ell (Q))}{M}$$
and we deduce that
\begin{align}
 \pi(a_n)-\pi(a_1)+\frac{O(\ell (Q))}{M} &\geq (\pi(a_n)-\pi(a_1)) \left( 1 + C(\varepsilon )\right)
\\
&\quad -  (\varepsilon + C(\varepsilon)) \sum_{j\in G}
 ( \pi(a_{j+1})-\pi(a_j) ). \nonumber
\end{align}
Therefore there exists a constant $C_1 (\varepsilon) >0$ such that
$$\sum_{j\in G} (\pi(a_{j+1})-\pi(a_j))\geq C_1 (\varepsilon) (\pi(a_n)-\pi(a_1)).$$
Since $\pi(a_{j+1})-\pi(a_j)$ and $\pi(a_n)-\pi(a_1)$ are
comparable, respectively, to $\ell(Q_j)$ and $\ell(Q)$, we obtain
(c). \hfill$\qed$

\begin{rem}\label{remar1}
Taking $u=-e$ in part (d), we obtain that for any $j=1,2, \ldots$,
one has
\begin{align*}
\|V(Q_j^*) \|^2 &= \| V(Q_j^*) - Me    \|^2 + M^2 + 2M \langle  V(Q_j^*)-M e,  e \rangle\\
&\leq  (\varepsilon M + w(\ell (Q)))^2 + M^2 - 4 \varepsilon^2 M^2 /3 = \\
&= M^2 (1- \varepsilon^2 /3 ) + 2 w(\ell (Q)) \varepsilon M +
w(\ell (Q))^2 .
\end{align*}
Taking the parameters so that  $M \varepsilon >20$ and assuming
$w(\ell (Q)) \leq 1$, we easily deduce that $\|V(Q_j^*) \| <M C_2
(\varepsilon)$, where $ C_2(\varepsilon) = (1-\varepsilon^2
/6)^{1/2} < 1$. So starting from a cube $Q$ so that the discrete
gradient satisfies $\|V(Q)\|=M$, Proposition~\ref{propo1} provides a
family $\{Q_j^* \}$ of dyadic subcubes  of $Q$ satisfying
$\|V(Q_j^*)\| < M C_2(\varepsilon) < M$, as well as properties
(a)-(c). This fact will be crucial in the proof of
Theorem~\ref{teo1}.
\end{rem}

\begin{rem}\label{remar2}
If $0< \varepsilon < 1$ is taken close to $1$, our arguments
 show that one can replace in part (d) the constant $2\varepsilon
 /3$ by a constant close to $1$. This means that $V(Q_j^*)$ lies
 in a cone with vertex at the point $M e$ and small aperture.
 This fact will be used in the proof of certain refinements of Theorem~\ref{teo1}.
\end{rem}

We can now proceed to prove Theorem~\ref{teo1}.

\emph{Proof of Theorem~\ref{teo1}.} Let us first discuss part (a).
Let $f$ be a function in the Zygmund class and let $V(Q)$ denote its
discrete gradient at a cube $Q$. We can assume that the function
$w(\delta)$ defined in Proposition~\ref{propo1} satisfies $w(\delta)
\leq 1$ for any $\delta >0$. Corollary~\ref{coro1} tells that we can
restrict our attention to the behavior of the discrete gradient on
dyadic cubes. We can assume that condition \eqref{derinfi} is
satisfied and that $V(Q^1)=0 \in \Rd$ where $Q^1$ is the unit cube
in $\Rd$. We will construct a Cantor type set on which the function
$f$ has bounded divided differences. The generations $\A(n)$ of the
Cantor type set will be defined recursively. Fix a number $0 <
\varepsilon < 1$, say $\varepsilon = 1/2$. Fix a large number $M>M_0
(1)$, where $M_0 (\delta)$ is the constant appearing in
Proposition~\ref{propo1}. The first generation $\A(1)$ consists of
the unit cube $Q^1$. The second generation is constructed as
follows. Consider the maximal dyadic cubes $\{R_j\}$ contained in
$Q^1$ so that $\|V(R_j)\| \geq M$. The maximality and the estimate
$w(\delta) \leq 1$ give that $\|V(R_j)\| \leq M+ 1$ for any
$j=1,2,\ldots $. Let $e_1$ be the first vector of the canonical
basis of $\Rd$ and let $L$ be the segment given by $L=\{t e_1 : 0
\leq t \le 1 \}$ and consider the subfamily $\mathcal{R}$ of $\{R_j
\}$ formed by those $R_j$ which have a non empty intersection with
$L$. By \eqref{derinfi} every point of $L$, except at most for a set
of zero length, is contained in a cube of the family $\mathcal{R}$.
Let $\Pi$ be the projection over the first coordinate, that is, $\Pi
(x) = \langle x , e_1 \rangle e_1$, $x \in \Rd$. Since $\{ \Pi (R_j)
: R_j \in \mathcal{R} \}$ are pairwise disjoint, for any cube $Q
\subset \Rd$ we have
$$
\sum \ell (R_j) \leq \ell (Q) \, ,
$$
where the sum is taken over all cubes $R_j \in  \mathcal{R}$
contained in $Q$. Assume that $N=\varepsilon M = M/2 >20$. Apply
Proposition~\ref{propo1} in each cube $R \in \mathcal{R}$ to obtain
a family of dyadic cubes $\mathcal{Q}(R) =\{Q_j \}$ contained in $R$
satisfying conditions (a)- (d). In particular by
Remark~\ref{remar1}, we have  $\|V (Q_j)\| < M$ for any $Q_j \in
\mathcal{Q}(R) $. The second generation $\A(2)$ of the Cantor type
set is defined as the union of the families of cubes $\{
\mathcal{Q}(R) : R \in \mathcal{R} \}$. Observe that if $x \in Q$
for some $Q \in \A(2)$ and $n \leq -\log_2 \ell(Q)$, property (b) of
Proposition~\ref{propo1} gives that $\|V(Q_n (x)) \| \leq M+1+N \leq
2M+1$. Property (a) of Proposition~\ref{propo1} tells that $\ell(Q)
< 2^{-N} $ for any $Q \in \A(2)$,  while Property (c) gives that
$$
\sum \ell (Q) \geq C_0^{-1} \, ,
$$
where the sum is taken over all cubes $Q \in \A(2)$. Moreover
 if $R$ is a cube in $\Rd$
we have $\sum \ell (Q) \le C_0 \ell (R)$, where the sum is taken
over all cubes $Q \in \A(2)$ contained in $R$. The construction
continues inductively. Assume we have defined the generation
$\A(n)$ of cubes satisfying $\|V(Q) \| \le M$ for any $Q \in
\A(n)$. In each $Q \in \A(n)$ we act as in the first step, that
is, we consider the family $\{R_j  \}$ of maximal dyadic cubes
$R_j  = R_j (Q)$ contained in $Q$ such that $\|V(R_j)\| \geq M$.
\linebreak  Let $a=a(Q)$ be the origin of $Q$. Consider the
segment $L=L(Q)=\{ a + t e_1 \colon \break  0 \leq t \le \ell (Q)
\}$ and the subfamily $\mathcal{R} (Q)$ of $\{R_j \}$ formed by
those $R_j$ which have a non empty intersection with $L$. Finally
in each $R \in \mathcal{R} (Q)$ apply Proposition~\ref{propo1} to
obtain a family of dyadic cubes $\mathcal{Q}(R) =\{Q_j \}$
contained in $R$ satisfying conditions (a)-(d). The generation
$\A(n+1)$ is given by the cubes $\{Q_j \colon  Q_j \in \mathcal{R}
(Q), \break Q \in \A(n) \}$. Properties (a) and (c) of
Proposition~\ref{propo1} give conditions (a)-(c) in
Lemma~\ref{lema1} with the parameters $s=1$, $\eta_0 = 2^{-N}$,
$K_0 = C_0^{-1} $ and $\widetilde K = C_0 $. Hence the Hausdorff
dimension of the set
$$
E=E(M) = \bigcap_{n=1}^\infty \bigcup_{Q\in \A(n)} Q
$$
is bigger than $1 - \log_2 (C_0) /N$. Since for any point $x \in E$
we have  $\|V(Q_n (x)) \| \leq 2M + 1$, for any $n=1,2, \ldots $, we
deduce that at any point in $E=E(M)$ the function $f$ has bounded
divided differences. Since $N=M/2$ can be taken large the proof of
(a) is completed.

We now discuss part (b) of Theorem~\ref{teo1}. Since the proof is
similar to the previous one we will only sketch the argument. Let
$f$ be a function in the small Zygmund class. Corollary~\ref{coro1}
tells that it is enough to show that the Hausdorff dimension of the
set of points where the discrete gradient converges is bigger or
equal to 1. We will construct a Cantor type set on which the
discrete gradient converges. The generations $\A(n)$ of the Cantor
type set will be defined recursively. The main idea is to observe
that in Proposition~\ref{propo1} if the function $f$ is in the small
Zygmund class, at small scales one can use small parameters $M$.
More precisely, the proof of Proposition~\ref{propo1} gives that one
can take $M_0 = M_0 (\delta)$ with $M_0 (\delta ) \to 0$ as $\delta
\to 0$. So, choose an increasing sequence of integers  $N(n) \to
\infty$ as $n \to \infty$ such that
\begin{equation}\label{38}
\sum_n M_0 (2^{-N(n)}) < \infty \, .
\end{equation}
The first generation $\A(1)$ consists of the unit cube $Q^1$.
Assume by induction that we have defined the generation $\A(n)$
satisfying $\ell (Q) \leq 2^{-N(n)}$ for any $Q \in \A(n)$. In
each $Q \in \A(n)$ we use a slight variation  of the construction
of the generations in part (a) with the purpose that the cubes of
generation $n+1$ contained in $Q$ have sidelength smaller than
$2^{- N(n+1)}$ and still have discrete gradient close to $V(Q)$.
Fix $Q \in \A(n)$ and argue as in part (a) with the ball centered
at the origin of radius $M$ replaced by the ball centered at
$V(Q)$ of radius $2 M_0 (2^{-N(n)})$, to obtain an intermediate
family of cubes $\{R_j \}$ so that $\|V(R_j) - V(Q) \| \leq 2 M_0
(2^{-N(n)})$ for any $j=1,2, \ldots$. In each intermediate cube
$R_j$ we can repeat this process finitely many times to obtain a
family of cubes $\mathcal{R} (Q)$ such that $\ell(R) \leq 2^{-
N(n+1)}$ and $\|V(R) - V(Q) \| \leq 2 M_0 (2^{-N(n)})$ for any $R
\in \mathcal{R} (Q)$. The family $\A(n+1)$ will be formed by the
dyadic cubes in $\cup \mathcal{R} (Q)$, where the union is taken
over all cubes $Q \in \A(n)$. Consider the set $E= \bigcap_n
\bigcup_{Q\in \A(n)} Q$. As before Lemma~\ref{lema1} tells that
the Hausdorff dimension of $E$ is bigger or equal to one. Also
condition \eqref{38} tells that for any $x \in E$ the sequence
$V(Q_n (x))$ converges as $n \to \infty$.
 \hfill$\qed$

\section{Theorem~\ref{teo2}}

The purpose of this section is to construct a function $f$ in the
small Zygmund class $\lambda_*(\Rd)$ for which the set $E(f)$ has
Hausdorff dimension 1. Let us first explain the main idea of the
construction. Fix a sequence of positive numbers $0\leq
\varepsilon_k \leq 1$, with $\lim \varepsilon_k=0$ such that $\sum
\varepsilon_k^2=\infty$. The function $f$ will be constructed as
$f=\sum g_k$, where $\{g_k\}$ will be a sequence of smooth functions
defined recursively with $\sum ||g_k ||_\infty < \infty$. Certain
size estimates of $g_k$ and its gradient will give that $f$ is in
the Small Zygmund Class. Assume the functions $\{g_j : j =0, \ldots,
k \}$ have been defined. The idea is to construct the function
$g_{k+1}$ in such a way that $\nabla g_{k+1} (x)$ is \emph{almost
orthogonal} to $\nabla \sum_{j=1}^k g_j (x)$ and $|\nabla g_{k+1}
(x)|$ is comparable to $ \varepsilon_{k+1}$ for \emph{most} points
$x\in \Rd$. Since one can not hope to achieve both requirements at
all points $x\in\Rd$, an exceptional set appears and \emph{one
dimensional} estimates of the size of this set will be needed. The
construction is easier in dimension $d=2$ because the boundary of a
square has dimension 1 while when $d>2$ some extra requirements will
be needed to obtain the convenient one dimensional estimates.

Pick $g_0\equiv 0$. Fix an integer $k>0$ and  assume by induction
that $f_k=\sum_{j=0}^k g_j$ has been defined. Assume also by
induction that $\nabla f_k$ is uniformly continuous in $\Rd$ and
that
\begin{equation*}
M_k =\sup_{\Rd} \|\nabla f_k \| <\infty.
\end{equation*}
The construction of the function $g_{k+1}$ is presented in several
steps. The reader only interested in dimension $d=2$ could skip
paragraphs 3,4 and 5.

\noindent $\textbf{1.}$ Choose a positive number $\eta_k
>0$ with $ \eta_k <2^{-k}/M_k $ and $\eta_k \le \varepsilon_k^3$
and find $N_k>0$ such that
\begin{equation}\label{eqq41}
\|\nabla f_k (z) -\nabla f_k (w)\| < \eta_k
\end{equation}
if $\|z-w\|\leq 2^{-N_k}$.

\medskip

\noindent $\textbf{2.}$ Consider the collection $\mathcal{Q}=\{Q_j
\colon j=1,2,\dots\}$ of pairwise disjoint open-closed dyadic
cubes of generation $N_k$. So $\Rd=\bigcup Q_j$, $\ell
(Q_j)=2^{-N_k}$. Let $a_j$ be the center of $Q_j , j=1,2,\dots$.
All cubes in the construction will be in this family or in a
family obtained by translating $\{ Q_j \colon j=1,2,\dots\}$ by a
fixed vector. When $d=2$ let $e_1(a_j)$  be a unit vector
orthogonal to $\nabla f_k (a_j)$.

\medskip

\noindent $\textbf{3.}$  This paragraph only applies when the
dimension $d>2$. We will choose a set of unit vectors $\{e_\ell
(a_j)\colon \ell =1,\dots, d-1\}$ associated to each cube $Q_j$ ,
$j= 1,2,...$ which are \emph{essentially} an orthonormal basis of
the hyperplane orthogonal to $\nabla f_k (a_j)$. More precisely,
we claim that there exists a constant $C>0$ only depending on the
dimension, such that for each cube $Q_j \in \mathcal{Q}$,
$j=1,2,\dots$ one can choose unit vectors $\{e_\ell (a_j)\colon
\ell =1,\dots, d-1\}$ satisfying the following three properties:
\begin{enumerate}[(a)]
\item $| \langle \nabla f_k (a_j) , e_\ell (a_{j}) \rangle | \leq
C\, \eta_k^{1/2}$ for $\ell=1, \ldots, d-1$.
\item For any
collection $\{Q_{j(\ell)}\colon \ell \!=\!1,\dots,d -\!1\}$ of
non-necessarily distinct $d-\!1$ cubes in the family $\mathcal{Q}$
with $\bigcap_{\ell=1}^{d-1} \overline{Q_{j(\ell)}} \neq~\emptyset$,
the corresponding vectors $\{e_\ell (a_{j(\ell)}) : \ell = 1, 2,
\ldots , d-1 \}$ form a $C$-Riesz set, that is, for any set
$\{\alpha_\ell\}$ of real numbers, one has
\begin{equation*}
C^{-1} \sum_{l=1}^{d-1} |\alpha_\ell|^2 \leq \left \|
\sum_{\ell=1}^{d-1} \alpha_\ell\, e_\ell \,(a_{j(\ell)}) \right
\|^2 \leq C \sum_{\ell=1}^{d-1} |\alpha _\ell |^2.
\end{equation*}
\item The angle between any coordinate axis and any $e_\ell
(a_j)$, $\ell =1,\dots, d-1$, $j=1,2,\dots$, is between, say,
$\pi/6$ and $\pi/3$.
\end{enumerate}

The first condition tells that the unit vectors $\{e_\ell
(a_j)\colon \ell =1,\dots, d-1\}$ are, up to a small error,
orthogonal to the gradient  $\nabla f_k (a_j)$. Condition (b)
tells that the unit vectors act, up to constants, as an
orthonormal basis.  Property (c) tells that any hyperplane with
normal unit vector $e_\ell (a_j)$ is far from being parallel to
any face of any dyadic cube. These facts will be used later. Let
us explain the choice of the vectors $\{e_\ell (a_j)\colon \ell
=1,\dots, d-1\}$.

Let $\A$ be the set of indexes $j$ so that $\|\nabla f_k (a_j)\|\geq
\eta_k^{1/2}$. If $\A$ is empty one can take $\{e_\ell (a_j) \colon,
\ell=1,\dots, d-1\}$ to be any fixed orthonormal system satisfying
(c). Then (a) and (b) with $C=1$ are satisfied . So, assume $\A$ is
non empty. We first construct the unit vectors corresponding to
indexes in $\A$. Observe that by \eqref{eqq41}, one has
\begin{equation*}
\left  \| \frac{\nabla f_k (a_j)}{\| \nabla f_k (a_j)\|} -
    \frac{\nabla f_k (a_{j^{'}})}{\| \nabla f_k (a_{j^{'}})\|}
    \right \| <2\eta_k^{1/2} \; \textrm{ if }\;  j, j{'} \in \A \;\textrm{ and } \;\overline{Q_j} \cap \overline{Q_{j'}} \neq \emptyset .
\end{equation*}
Hence the orthogonal hyperplanes $M_j=\{ x\in \Rd \colon \langle x ,
\nabla f_k (a_j) \rangle =0\}$ deviate smoothly for contiguous $j$'s
in $\A$ and for every $j\in \A$ one can choose an orthonormal basis
$\{e_\ell (a_j) \colon \ell=1,\dots, d-1\}$ of $M_j$ which for any
$\ell =1,\dots, d-1$ satisfies
\begin{equation}\label{eqq42}
\| e_\ell (a_j)-e_\ell (a_{j'})\| < C\,\eta_k^{1/2} \; \textrm{ if
} \; j, j' \in \A \; \textrm{ and } \;\overline{Q_j}\cap
\overline{Q_{j'}} \neq \emptyset.
\end{equation}
Moreover when $d>2$, the dimension of $M_j$ is $d-1 \geq 2$ and one
can take $\{e_\ell (a_j) \colon \ell=1,\dots,d-1\}$ verifying also
(c). Observe that since $\{e_\ell (a_j) \colon \break  \ell=1,\dots,d-1\}$
is an orthonormal basis of $M_j$, the continuity property
\eqref{eqq42} gives
 (b) for indexes $j$ in the set $\A$. We now need to construct
$\{e_\ell (a_j) \colon \ell=1,\dots,d-1\}$ when $j\notin \A$. We
argue recursively and start with the family $\A_1$ of indexes $j\notin \A$ so that
$\overline{Q_j} \cap \overline{Q_{j'}} \neq \emptyset$ for some
index $j'\in \A$. Observe that since $j\notin \A$, property (a) with
$C=1$ will be satisfied as soon as the vectors $\{e_\ell (a_j)\colon
\ell =1,\dots, d-1\}$ are chosen of unit norm. We first choose the
vector $e_1 (a_j)$. For every set of non-necessarily distinct
indexes $j(2),\dots, j(d-1)\in \A$ such that $\overline{Q_j}\cap\,
\bigcap_{\ell=2}^{d-1} \overline{Q_{j(\ell)}}\neq \emptyset$, let
$\widetilde M\!=\!\widetilde M (j, j(2), \dots, j(d-\!1))$ be the subspace
in $\Rd$ generated by the vectors \linebreak $\{e_\ell (a_{j(\ell)}) \colon
\ell=2,\dots,d-1\}$. Observe that fixed $j$,  the number of possible
choices of indexes and hence of subspaces of this form, is bounded
by a
constant only depending on the dimension. Hence one can choose
the unit vector $e_1(a_j)$ to be at a fixed positive angle, only
depending on the dimension, to any such subspace $\widetilde M$ and
also verifying (c). This guarantees (b) for the vectors $\{e_1
(a_j), e_\ell (a_{j(\ell)})\colon \break  \ell =2,\dots, d-1\}$. Once
$e_1(a_j)$ has been chosen, we choose $e_2 (a_j)$ similarly. Now the
set of indexes will be of the form $j(1), j(3), \dots, j(d-1)$ where
$j(3),\dots, j(d-1)\in \A$ and either $j(1) \in \A$ or $j(1) = j$,
so that
\begin{equation*}
\overline{Q_j} \cap \bigcap_{\ell= 1 \atop \ell \neq 2}^{d-1}
\overline{Q_{j(\ell)}}\neq \emptyset
\end{equation*}
and we choose $e_2(a_j)$ to be a unit vector at a fixed positive
angle to any subspace in $\Rd$ generated by $\{e_\ell
(a_{j(\ell)})\colon \ell=1,3,\dots, d-1\}$ and also verifying (c).
In this way, for an index $j\notin \A$ so that $\overline{Q_j} \cap
\overline{Q_{j'}}\neq \emptyset$ for some $j'\in \A$, the unit
vectors $\{e_\ell (a_j)\colon \ell=1,\dots, d-1\}$ are chosen. So,
the unit vectors $\{e_\ell (a_j)\colon \ell=1,\dots, d-1\}$ are
constructed for the family $\A_1$. 
Next, using the same procedure, one constructs the unit vectors
$\{e_\ell (a_j)\colon \ell=1,\dots, d-1\}$ for the set $\A_2$ of
indexes $j$ such that $\overline{Q_j} \cap \overline{Q_{j'}}\neq
\emptyset$ for some $j' \in \A_1$. The construction continues
inductively.

\medskip

\noindent $\textbf{4.}$ This paragraph only applies when the
dimension $d>2$. In that case we will need $d-1$ different
collections of translated dyadic cubes in order to obtain the
convenient one dimensional estimates. The idea is that the
exceptional set will be contained in an intersection of boundaries
of cubes which are in different translated collections. This fact
will be used later when computing the dimension of the exceptional
set.

Denote $Q_j^{(1)}=Q_j$, $\; j=1,2,\dots$, and for $\ell=2, \dots,
d-1$, let $\{ Q_j^{(\ell)}\colon j=1,2,\dots, \}$ be the collection
of pairwise disjoint open-closed cubes of lengthside $\ell
(Q_j^{(\ell)})=\ell (Q_j)=2^{-N_k}$ and center $a_j^{(\ell)}$ given
by
\begin{equation*}
a_j^{(\ell)}= a_j +\frac{(\ell -1)}{10 (d-1)} 2^{-N_k} \mathbf{1}
\quad , j =1, 2, \ldots \, ,
\end{equation*}
where $\mathbf{1}= d^{-1/2}(1,\dots, 1)\in \Rd$. In other words
consider $Q_j ^{(\ell)}=\lambda (\ell)+Q_j$ where $\lambda (\ell)=
(\ell-1)2^{-N_k} \mathbf{1} /10(d-1) $, $\ell=2, \ldots, d-1$.
Observe that for any $\ell =1, \dots, d-1$, the cubes
$\{Q_j^{(\ell)} \colon j=1,2,\dots\}$ are pairwise disjoint,
$\Rd=\bigcup_j Q_j^{(\ell)}$ and moreover cubes in different
families $\{Q_j^{(\ell)} \colon j=1,2,\dots\}$ corresponding to
different indexes~$\ell$, intersect nicely. See
Figure~\ref{figura2}. More precisely for any set of distinct
indexes $\F \subset \{1,\dots, d-1\}$ of cardinality $n$, each set
of the form
\begin{equation*}
\bigcap_{\ell\in \F} \partial  Q_{j(\ell)}^{(\ell)}
\end{equation*}
is contained in at most $C(d)$ distinct $(d-n)$-dimensional planes
of $\Rd$ parallel to a coordinate hyperplane. Here $\partial Q$
means the boundary of the cube $Q$. To see this, observe that if
$\ell_1\neq \ell_2$, the choice of the~ centers of
$\{Q_j^{(\ell)}\}$ guarantees that two parallel faces of
$Q_{j(\ell_1)}^{(\ell_1)}$ and $Q_{j(\ell_2)}^{(\ell_2)}$, with
$\ell_1 \neq \ell_2$ never meet. So the points in $\bigcap_{\ell \in
\F}
\partial Q_{j(\ell)}^{(\ell)}$ must lie into at most $C(d)$ distinct $ (d - \# \F)$-dimensional planes of $\Rd$.

 \begin{figure}
\begin{center}
\epsfig{figure=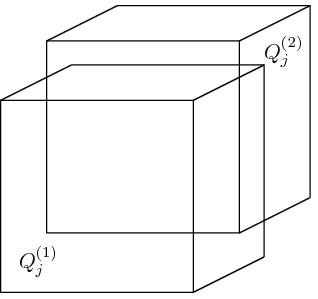} \caption{The cubes $Q_j^{(1)}$ and
$Q_j^{(2)}$ in $\R^3$.}\label{figura2}
 \end{center}
\end{figure}
\medskip

\noindent $\textbf{5.}$ This paragraph applies only when the
dimension $d>2$. Let $\pi_\ell (a_j)$ denote a hyperplane in $\Rd$
with  normal unit vector $e_\ell (a_j)$. The content of this
paragraph is the existence of a fixed constant $C(d)$ with the
following property. For any collection of $d-1$ distinct cubes
$\{Q_{j(\ell)} \colon \ell=1,\dots, d-1\}$ of the family
$\mathcal{Q}$ such that $\bigcap_{\ell =1}^{d-1}
\overline{Q_{j(\ell)}} \neq \emptyset$ and for any collection $ \F
\subset \{1, \dots,d-1\}$ of distinct indexes, the set
\begin{equation*}
\bigcap_{\ell\in \F} \partial Q_{j(\ell)}^{(\ell)} \; \cap \,
\bigcap_{\ell \in \{1,\dots, d-1\} \setminus \F} \pi_ \ell
(a_{j(\ell)})
\end{equation*}
is contained in at most $C(d)$ lines in $\Rd$.

This follows from the fact that the set under consideration is
contained in a union of at most $C(d)$ sets each of them being an
intersection of $d-1$ hyperplanes so that the angle between two of
them is bigger than a fixed constant only depending on the
dimension (and hence this set is contained in a line). To see this
observe that by paragraph~4, the points in $\bigcap_{\ell \in \F}
\partial Q_{j(\ell)}^{(\ell)}$ must lie into at most $C(d)$ distinct $(d- \#
\F)$-dimensional planes parallel to a coordinate hyperplane of
$\Rd$ while by property (c) in Paragraph 3, the angle between the
hyperplanes $\{\pi_\ell (a_{j(\ell)}) \colon \ell \in \{1,\dots,
d-1\} \setminus \F \}$ and any
 coordinate hyperplane is bigger than $\pi/6$ and by (b) of
paragraph~3 the angle between two different $\pi_\ell (a_{j(\ell)})$
is also bigger than a fixed constant.

\medskip

\noindent $\textbf{6.}$ In this  paragraph standard bump functions
adapted to the collection $\{Q_j^{(\ell)} \colon j=1,2,\dots\}$
are constructed. Pick a small positive number $\alpha_k >0$,
$\alpha_k <2^{-N_k-2}$ and a function $w_k \colon \Rd
\longrightarrow \R$, $0\leq w_k\leq 1$ with
\begin{equation*}
\begin{split}
&w_k \equiv 0 \quad \textrm{ on } \quad \Rd\setminus [-2^{-N_k
-1},2^{-N_k -1}]^d \, ,
\\*[3pt] &w_k \equiv 1 \quad \textrm{ on } \quad [-2^{-N_k
-1}+ \alpha_k , 2^{-N_k -1} -\alpha_k]^d  \, ,  \\*[3pt] &
\sup_{x\in \Rd} \alpha_k \| \nabla w_k (x)\|=C(d)<\infty.
\end{split}
\end{equation*}
Recall that the center of the cube $Q_j^{(\ell)}$ was denoted by
$a_j^{(\ell)}$ and consider the function
$
w_j^{(\ell)} (x) =w_k (x-a_j^{(\ell)}), \quad x\in \Rd.
$ 
Hence $w_j^{(\ell)} \equiv 0$ on $\Rd \setminus Q_j^{(\ell)}$,
$w_j^{(\ell)} \equiv  1$ on $ (1- \alpha_k 2^{N_k+1})
Q_j^{(\ell)}$ and $\alpha_k  \|\nabla w_j^{(\ell)} (x) \|\leq
C(d)$ for any $x\in \Rd$.

\medskip

\noindent $\textbf{7.}$  In this paragraph a one dimensional
function with \emph{small} size and \emph{large} derivative is
constructed. Pick a small positive number $\sigma_k
>0$ with  $\sigma_k <\varepsilon_k \alpha_k  \eta_k /(1 + M_k)$.
Pick an integer $n_k
>2\sigma_k^{-1}$ and a small positive number $\beta_k
<C(d)2^{-N_k}/ n_k$. Consider a smooth one dimensional periodic
function $\phi_k \colon \R\to~\R$, satisfying  $\phi_k (x+1)=\phi_k
(x)$, for any  $x\in \R$, such that
\begin{equation*}
\begin{split}
& \sup_{x\in \R} |\phi_k (x)|\leq \sigma_k\\*[3pt] & \sup_{x\in
\R} |\phi'_k (x)|\leq \varepsilon_k\
\end{split}
\end{equation*}
and such that $|\phi^{'}_k (x)|=\varepsilon_k$ for \emph{most}
points $x\in \R$. More concretely we require that the set $ \{x \in
[-1,1] \colon |\phi^{'}_k (x) |\neq \varepsilon_k\} $ can be covered
by $4n_k$ intervals $\{J_i\}$ of length $\beta_k$. This can be done
by smoothing the function $ \psi_k (t)=\varepsilon_k \, \inf
\{|t-i/n_k|\colon i\in \Z\}$. Observe that $\|\psi_k\|_\infty \leq
\varepsilon_k/n_k<\varepsilon_k \sigma_k<\sigma_k$ and $|\psi^{'}_k
(x)|=\varepsilon_k$ if $x\neq i/n_k$ for any $i\in \Z$. See Figure~\ref{figura3}.
Observe that the lengths of the intervals $\{J_i\}$ where the
regularization is performed can be taken as small as desired.
\smallskip

 \begin{figure}
\begin{center}
\epsfig{figure=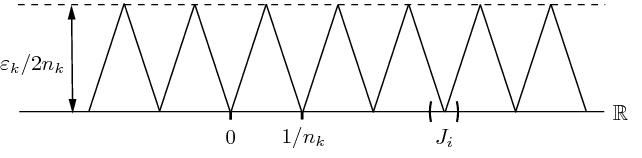} \caption{The graph of the function
$\psi_k$.}\label{figura3}
 \end{center}
\end{figure}

\noindent $\textbf{8.}$ As explained before, the main idea is to
construct a function $g_{k+1}$ whose \linebreak  gradient is
\emph{essentially} orthogonal to $\nabla f_k$, where
$f_k=\sum_{j=0}^k g_j$. To achieve this, roughly speaking, in each
cube $Q_j=Q_j^{(1)}$ the function $g_{k+1}$ will look like \linebreak
$\sum_{\ell=1}^{d-1} \phi_k (\langle x-a_j ,  e_\ell (a_j) \rangle
)$, because the gradient of this function is a linear combination
of the vectors $\{e_\ell (a_j) \colon \ell=1,\dots, d-1\}$ which
by (a) of paragraph 3 are almost orthogonal to $\nabla f_k (a_j)$.
The functions $\{w_j^{(\ell)} \}$ of paragraph~6 will be used to
paste together the pieces corresponding to different cubes. More
concretely, consider the function
\begin{equation*}
g_{k+1} (x)=\sum_{\ell=1}^{d-1} \sum_j w_j^{(\ell)} (x) \phi_k
(\langle x-a_j , e_\ell (a_j) \rangle) \, , x \in \R^d \, ,
\end{equation*}
and $f_{k+1}=f_k+g_{k+1}$. Observe that for every  $\ell=1,\dots,
d-1$ there is at most one single non-vanishing term in the inner
sum. Also observe that for every  $\ell=1,\dots, d-1$ we use the
center $a_j$ of $Q_j=Q_j^{(1)}$ and the vectors $\{e_\ell (a_j) :
\ell= 1, \ldots , d \}$ corresponding to $\nabla f_k (a_j)$. Since
$\|w_j^{(\ell)}\|_\infty \leq 1$, $\|\nabla w_j^{(\ell)}\|_\infty
\leq C(d)/\alpha_k$, $\|\phi_k\|_\infty\leq \sigma_k$ and
$\|\phi^{'}_k \|_\infty \leq \varepsilon_k$, we deduce that
\begin{equation*}
\begin{split}
& \sup_{x\in \Rd} |g_{k+1}(x)| \leq (d-1)\sigma_k \, , \\*[3pt] &
\sup_{x\in \Rd} |\nabla g_{k+1}(x)| \leq C\left (
\frac{\sigma_k}{\alpha_k}+\varepsilon_k\right ) \leq 2
C\varepsilon_k \, ,
\end{split}
\end{equation*}
because $\sigma_k$ was chosen so that
$\sigma_k<\varepsilon_k\alpha_k$. Here $C$ is a constant only
depending on the dimension.

All quantifiers needed to define $g_{k+1}$ have already appeared.
They are small numbers which have to be chosen in the right order,
so that at each step the new one is smaller than a certain function
of the previous ones. Let us summarize the order. Given
$\varepsilon_k
>0$ with $\sum \varepsilon_k^2 = \infty$, the quantifiers $\eta_k$,
$2^{-N_k}$ are chosen to satisfy \eqref{eqq41}. Given these three
quantities, the small number $\alpha_k$ is chosen in paragraph 6
when the bump functions $w_j^{(\ell)}$ are defined. Finally the last
three ones $\sigma_k$, $1/n_k$ and $\beta_k$ are chosen in paragraph
7 when the functions $\phi_k$ are introduced. We continue
recursively and we can assume that
\begin{equation*}
\sum_{j \ge k} \sigma_j \leq  2 \sigma_k
\end{equation*}

\medskip

\noindent $\textbf{9.}$ In this paragraph we show that $\nabla
g_{k+1} (x)$ is almost orthogonal to $\nabla f_k (x)$. This fact
will be crutial in the construction. Fix $x\in \Rd$. For $\ell
=1,\dots, d-1$, let $j(\ell)$ be the index so that $x\in
Q_{j(\ell)}^{(\ell)}$. Since $\|\nabla w_j^{(\ell)}\|_\infty \leq
C/\alpha_k$ and $\|\phi_k\|_\infty \leq\sigma_k$, we have
\begin{equation}\label{eqq95}
\nabla g_{k+1}(x) =\sum_{\ell =1}^{d-1} w_{j(\ell)}^{(\ell)} (x)
\phi^{'}_k (\langle x-a_{j(\ell)} , e_\ell (a_{j(\ell)}) \rangle )
\, e_\ell (a_{j(\ell)})+O(\sigma_k/\alpha_k).
\end{equation}
Here the notation $O(A_k)$ means a vector (or a quantity) whose
 norm (or modulus) is bounded by a fixed proportion, independent
 of $k$ and $x$, of $A_k$. Since $\bigcap_{\ell =1}^{d-1} Q_{j(\ell)}^{(\ell)} \neq~
\emptyset$, we have $\bigcap_{\ell=1}^{d-1} \overline{Q_{j(\ell)}}
\neq \emptyset$ and condition (a) of paragraph~3, tells that
\begin{equation*}
|\langle \nabla f_k (a_{j(\ell)}) , e_\ell (a_{j(\ell)}) \rangle
|\leq C\eta_k^{1/2}.
\end{equation*}
Since by \eqref{eqq41}, one has $\|\nabla f_k(x)-\nabla f_k
(a_{j(\ell)})\|\leq C\eta_k$, we deduce that \linebreak $|\langle
\nabla f_k(x) ,  e_\ell (a_{j(\ell)}) \rangle |\leq 2C \eta_k^{1/2}.
$ Since $\|\phi^{'}_k \|_\infty \leq \varepsilon_k$ from
\eqref{eqq95} we deduce that
\begin{equation*}
\langle \nabla f_k(x) , \nabla g_{k+1} (x) \rangle =O(\eta_k^{1/2}
\varepsilon_k)+O\big (\sup_{\Rd} |\nabla f_k|\sigma_k /\alpha_k
\big ).
\end{equation*}
Since $\sigma_k<\varepsilon_k\eta_k\alpha_k/\sup_{\Rd} |\nabla
f_k|$, we deduce that there exists a constant $C$ only depending on
the dimension such that for any $x\in \Rd$ one has
\begin{equation}\label{eqq96}
| \langle \nabla f_k(x) , \nabla g_{k+1} (x) \rangle | \le C
\varepsilon_k \eta_k^{1/2} .
\end{equation}

\medskip

\noindent $\textbf{10.}$  We have observed that $\|\nabla
g_{k+1}\|_\infty\leq C\varepsilon_k$. Next we will show that $||
\nabla g_{k+1} (x)||$ is comparable to $\varepsilon_k$ for all
points $x\in \Rd$ except possibly for a set of points $A_k$ for
which one has a certain one dimensional estimate. Here the one
dimensional information given by the choice of the vectors $\{e_\ell
(a_j)\}$ and the cubes $\{Q_j^{(\ell)}\}$ explained in paragraph~5
will be used.

Fix $x\in \Rd$. As before, for $\ell=1,\dots,d-1$, let $j(\ell)$
be the index so that $x\in Q_{j(\ell)}^{(\ell)}$. The formula
\eqref{eqq95} tells that
\begin{equation}\label{eqq440}
\begin{split}
\| \nabla g_{k+1} (x )\| &\geq \left \| \sum_{\ell=1}^{d-1}
w_{j(\ell)}^{(\ell)} (x) \phi^{'}_k \big ( \langle x-a_{j(\ell)} ,
e\, a_{j(\ell)}) \rangle \big) e_\ell (a_{j(\ell)})\right \| -
O(\sigma_k/\alpha_k)
\\*[5pt]
&\geq C\sum_{\ell=1}^{d-1} w_{j(\ell)}^{(\ell)} (x) \left
|\phi^{'}_{k} \big ( \langle x-a_{j(\ell)} ,  e_\ell (a_{j(\ell)})
\rangle \big) \right |-O(\sigma_k/\alpha_k)
\end{split}
\end{equation}
because, since $x\in\bigcap_{\ell=1}^{d-1}Q_{j (\ell)}^{(\ell)}$,
we have $\bigcap_{\ell=1}^{d-1}\overline{Q_{j(\ell)}^{(\ell)}}
\neq \emptyset$ and consequently,  by condition (b) in
paragraph~3, the vectors $\{e_\ell (a_{j(\ell)}) \colon \ell
=1,\dots, d-1\}$ form a Riesz set. Now recall that
$w_{j(\ell)}^{(\ell)}\equiv 1$ on $(1-2^{N_k+1} \alpha_k)
Q_{j(\ell)}^{(\ell)}$ and $\left |\phi^{'}_k  \big ( \langle
x-a_{j(\ell)} ,  e_\ell (a_{j(\ell)}) \rangle \big)  \right |
=\varepsilon_k$ \linebreak if $\langle x-a_{j(\ell)}, e_\ell
(a_{j(\ell)}) \rangle$ is not contained in any of the $4n_k$
intervals in $[-1,1]$ of length $\beta_k$ which appeared in the
construction of the function $\phi_k$ in paragraph~7. Let
\linebreak $J\subset [-1,1]$ be one of such intervals and let
$c(J)$ be its center. Let $\pi_\ell (J)=\pi_\ell (a_{j(\ell)})
(J)$ be the hyperplane of $\Rd$ given by
\begin{equation*}
\pi_\ell (J)=\{x\in \Rd \colon \langle x \, , e_\ell (a_{j(\ell)})
\rangle = \langle a_{j(\ell)} ,  e_\ell (a_{j(\ell)}) \rangle
+c(J)\}
\end{equation*}
and let $\widetilde{\pi_\ell} (J)=\widetilde{\pi_\ell}
(a_{j(\ell)}) (J)$ be the neighbourhood of $\pi_\ell (J)$ given by
\begin{equation*}
\widetilde{\pi_\ell} (J)=\left \{x\in \Rd \colon  \left| \langle x
\, , e_\ell (a_{j(\ell)}) \rangle - \langle a_{j(\ell)} ,  e_\ell
(a_{j(\ell)}) \rangle
 -c(J) \right| <\beta_k\right \}.
\end{equation*}
Observe that if $x\notin \bigcup_J \widetilde{\pi_\ell}(J)$ then $
\left| \phi^{'}_k ( \langle x-a_{j(\ell)} \, ,  e_\ell
(a_{j(\ell)}) \rangle ) \right|=\varepsilon_k. $ Now given a
collection of cubes $\{Q_{j(\ell)}^{(\ell)} \colon \ell=1,2,\dots,
d-1\}$ such that $\bigcap_{\ell=1}^{d-1}
Q_{j(\ell)}^{(\ell)}\neq\emptyset$, let $A(\{Q_{j(\ell)}^{(\ell)}
: \ell = 1, \ldots  , d-1\})$ be the union over all possible
collections of distinct indexes $\F\!\subset\! \{1,\dots, d-1\}$,
of sets of the form
\begin{equation*}
A(\F)=\bigcap_{\ell\in \F} \left ( Q_{j(\ell)}^{(\ell)} \setminus
\big (1-2^{N_k+1}\alpha_k\big ) Q_{j(\ell)}^{(\ell)}
 \right ) \cap \, \bigcap_{\ell \in \{1,\dots,d-1\}\setminus \F} \,\cup_J\widetilde{\pi_\ell} (a_{j(\ell)})(J)\,.
\end{equation*}
Here the union $\cup_J$ is taken over the $4n_k$ intervals $J\subset
[-1,1]$ appearing in the construction of the function $\phi_k$.
Since $\bigcap_{\ell=1}^{d-1} Q_{j(\ell)}^{(\ell)} \neq \emptyset$
one has $\bigcap_{\ell=1}^{d-1} \overline{Q_{j(\ell)}} \neq
\emptyset$ and by paragraph~5 each set
\begin{equation*}
\bigcap_{\ell \in \F} \partial Q_{j(\ell)}^{(\ell)} \cap \,
\bigcap_{\ell\in \{1,\dots, d-1\}\setminus \F} \pi_\ell
(a_{j(\ell)})
\end{equation*}
is contained in at most $C(d)$ lines in $\Rd$. Since given
$\varepsilon_k, \eta_k, N_k>0$, the quantifiers $\alpha_k$ and
$\beta_k$ can be taken arbitrarely small, each set $A (\F)$ is
contained in a \emph{small} neighbourhood of  $C(d)$ lines. Let
$A_k=\bigcup A(\{ Q_{j(\ell)}^{(\ell)}: \ell = 1, \ldots , d-1
\})$ where the union is taken over all possible collections
$\{Q_{j(\ell)}^{(\ell)}: \ell=1, \ldots, d-1\}$ of cubes for which
$ \bigcap_{\ell=1}^{d-1} Q_{j(\ell)}^{(\ell)} \neq \emptyset.$
Since the set $A (\{Q_{j(\ell)}^{(\ell)} : \ell=1, \ldots,d-1\})$
is a finite union of $A(\F)$ and in each bounded set there are at
most $C(d)$ collections of cubes $\{Q_{j(\ell)}^{(\ell)}\}$, there
exists $\delta_k >0$, depending on the previous quantifiers
$\varepsilon_k$, $\eta_k$, $N_k$, $\alpha_k$, $\beta_k$ with
$\delta_ k \to 0$, as $k \to \infty$ and a collection of dyadic
cubes $\{R_j^{(k)}\}_j$ with sidelength smaller than $\delta_k$
such that $A_k \subset \cup_j R_j^{(k)}$ and
\begin{equation}\label{eqq107}
 \sum_{j\colon R_j^{(k)} \cap \,\{ \|x\| <N\} \neq \emptyset} \ell (R_j^{(k)})\leq
C(N, d)
\end{equation}
for any $N>0$. Assume now that $x\notin A_k$. As before, for any
$\ell\!=\!1,\dots,d-1$, let $j(\ell)$ be the index so that $x\in
Q_{j(\ell)}^{(\ell)}$. Let $\F$ be~the set of indexes $\ell$ in
$\{1,\dots,d\!-1\}$ so that $x\in (1\!-\!2^{N_k+1}\alpha_k)
Q_{j(\ell)}^{(\ell)}$ and hence $w_{j(\ell)}^{(\ell)}(x)\!=\!1$.
Hence $x\!\in\! \bigcap_{\ell \notin\F}
Q_{j(\ell)}^{(\ell)}\setminus (1-2^{N_k+1}
\alpha_k)Q_{j(\ell)}^{(\ell)}$ and since $x\notin A_k$, then
$x\notin \bigcap_{\ell\in \F}\cup_J \widetilde{\pi_\ell}
(a_{j(\ell)}) (J)$. Hence \linebreak there exists  $\ell \!\in\! \F$
with $x\notin \cup_J \widetilde{\pi_\ell} (a_{j(\ell)}) (J)$ and we
deduce that \linebreak $\left |\phi^{'}_k \big ( \langle
x-a_{j(\ell)} \, , e_\ell(a_{j(\ell)}) \rangle
\big)\right|=\varepsilon_k$. Therefore using \eqref{eqq440}  we get
\begin{equation*}
\begin{split}
 \|\nabla g_{k+1} (x)\| &\geq C \,\sum_{\ell=1}^{d-1} w_{j(\ell)}^{(\ell)} (x)\left  | \phi^{'}_k \big( \langle x-a_{j(\ell)}\, , e_\ell (a_{j(\ell)}) \rangle \big) \right |-O(\sigma_k/\alpha_k) \\*[5pt]
&\geq C\varepsilon_k -O(\sigma_k/\alpha_k).
\end{split}
\end{equation*}
Since $\sigma_k / \varepsilon_k \alpha_k$ was taken small we
deduce that
\begin{equation}\label{eqq108}
\|\nabla g_{k+1}(x) \|\geq C\varepsilon_k, \quad \textrm{ for } \;
x\in \Rd \setminus A_k.
\end{equation}

\medskip

\noindent $\textbf{11.}$ We now prove that there exists a fix
constant $C=C(d)>0$ such that if $k$ is sufficiently large one has
\begin{equation}\label{eqq119}
\|\nabla (f_k+g_{k+1})(x)\|^2 \geq \|\nabla f_k(x)\|^2
+C\varepsilon_k^2
\end{equation}
for any $x\in \Rd\setminus A_k$. To prove \eqref{eqq119}, write
\begin{equation*}
\|\nabla (f_k+g_{k+1})(x)\|^2 = \|\nabla f_k(x)\|^2 +  \|\nabla
g_{k+1}(x)\|^2 +2   \langle \nabla f_k(x) \, ,  \nabla g_{k+1}(x)
\rangle.
\end{equation*}
By \eqref{eqq96}, we have $\langle \nabla f_k(x) \, , \nabla
g_{k+1}(x) \rangle \!=\! O(\varepsilon_k \eta_k^{1/2})$ while
\eqref{eqq108} tells that $\|\nabla g_{k+1}(x)\| \break  \geq
C\varepsilon_k$ if $x\in \Rd\setminus A_k$. Therefore for any
$x\in \Rd\setminus A_k$ one has
\begin{equation*}
\|\nabla (f_k+g_{k+1})(x)\|^2 \geq \|\nabla f_k(x)\|^2 +
C\varepsilon_k^2 -O(\varepsilon_k\eta_k^{1/2}).
\end{equation*}
Since $\eta_k$ was chosen in paragraph 1 such that
$\eta_k<\varepsilon_k^3$ we deduce \eqref{eqq119}

\medskip

\noindent $\textbf{12.}$ In this paragraph we define the function
$f$ in $\lambda_* (\Rd)$ whose divided differences are unbounded at
any point except for a set of Hausdorff dimension $1$. Since
$\|g_{k+1}\|_\infty \leq (d-1)\sigma_k$ and $\sum \sigma_k <\infty$
we can consider
\begin{equation*}
f=\sum_{k=0}^\infty g_k.
\end{equation*}
The function $f$ is continuous in $\Rd$. In the next paragraphs, it
will be shown that $f\in \lambda_* (\Rd)$, $f$ has unbounded divided
differences at any point of the set $\Rd \setminus A$, where
\begin{equation*}
A=\bigcap_{j=1}^\infty  \bigcup_{k\geq j} A_k,
\end{equation*}
and that $A$ has $\sigma$-finite length.

\medskip

\noindent $\textbf{13.}$ We first show that $f\in \lambda_*(\Rd)$.
We will show that the functions $f_k=\sum_{j=0}^k g_j$ are uniformly
in the small Zygmund class. Fix an integer $k>0$. Let $h\in \Rd$ and
let $\{N_j \}$ be the quantifiers appearing in paragraph~1. We use
the notation $\Delta_2 f(x,h)= f(x+h) + f(x-h) - 2f(x)$,  for $x, h
\in \R^d$.  We distinguish two possible situations
\begin{enumerate}[(a)]
\item Assume $\|h\|\!<\!2^{-N_k}$. Since \eqref{eqq41} tells that
$\|\nabla f_k (z) -\nabla f_k(w)\|<\eta_k$ if
$\|z-w\| \break <2^{-N_k}$,
we have $|\Delta_2 f_k (x,h)|\leq \eta_k \|h\|$. On the other hand
since $\|\nabla g_{k+1}\|_\infty\leq C\varepsilon_k$, we deduce that
\begin{align*}
|\Delta_2 g_{k+1} (x,h)| & \leq  |g_{k+1} (x+h) -g_{k+1}(x)|+|g_{k+1}
(x)  -g_{k+1}(x-h)|\\*[4pt]
&\leq 2C\varepsilon_k \|h\|.
\end{align*}
 So, we
obtain  that
\begin{equation}\label{eqq1311}
|\Delta_2 f_{k+1} (x,h)|\leq (\eta_k +2C\varepsilon_k)\|h\| \, .
\end{equation}
\medskip

\item Assume $2^{-N_{j+1}} <\|h\|\leq 2^{-N_j}$ for some $j<k$. We
have
\begin{equation*}
|\Delta_2 g_{k+1} (x,h) | \leq 4 \| g_{k+1}\|_\infty \leq 4(d-1)
\sigma_k \leq \eta_k 2^{-N_k} <\eta_k \|h\|.
\end{equation*}
Therefore we have
\begin{equation*}
|\Delta_2 f_{k+1} (x,h) | \leq | \Delta_2 f_{k} (x,h)|+ \eta_k
\|h\|.
\end{equation*}
Iterating this estimate we obtain that if $2^{-N_{j+1}}<\|h\|\leq
2^{-N_j}$, one has
\begin{equation*}
|\Delta_2 f_{k+1} (x,h) | \leq| \Delta_2 f_{j+1} (x,h)| + \left
(\sum_{i=j+1}^k \eta_i \right ) \|h\|.
\end{equation*}
Now,  case (a) applies to $\Delta_2 f_{j+1} (x,h)$ and by
\eqref{eqq1311} we obtain $ |\Delta_2 f_{j+1} (x,h)| \leq  (
\eta_j +2C\varepsilon_j) \|h\| $ . Hence, if $2^{-N_{j+1}}\leq
\|h\|\leq 2^{-N_j}$, we get
\begin{equation}\label{eqq13122}
|\Delta_2 f_{k+1} (x,h)| \leq  \left (\sum_{i=j}^k \eta_i
+2C\varepsilon_j\right ) \|h\|.
\end{equation}
\end{enumerate}
Now, since $\sum \eta_i<\infty$ and $\varepsilon_j \to 0$, we
deduce that $f$ is in $\lambda_*(\Rd)$

\medskip

\noindent $\textbf{14.}$ Let $A=\bigcap_j\bigcup_{k\geq j} A_k$. In
this paragraph we show that $f$ has unbounded divided differences at
the points of $\Rd\setminus A$. Let $z,x\in \Rd$. One has
\begin{equation*}
f(z)-f(x) =f_k(z)-f_k(x)+\sum_{j>k} ( g_j(z)-g_j(x) )
\end{equation*}
Since $\|g_j\|_\infty \leq (d-1) \sigma_{j-1}$, we deduce that
\begin{equation*}
| f(z)- f(x) - (f_k(z)-f_k(x) ) |  \leq 2 (d-1)\Big (\sum_{j>k}
\sigma_j \Big ).
\end{equation*}
On the other hand, applying \eqref{eqq41} one gets
\begin{equation*}
| f_k(z)-f_k(x)- \langle \nabla f_k(x), z-x \rangle |\leq 2\eta_k
\|z-x\|
\end{equation*}
if $\|z-x\|\leq 2^{-N_k}$. Hence,
\begin{equation*}
| f(z)-f(x)- \langle \nabla f_k(x), z-x \rangle |\leq 2\eta_k
\|z-x\|+2(d-1)\sum_{j>k}\sigma_j \, ,
\end{equation*}
 if $\|z-x\|\leq 2^{-N_k}$. Since $\sum_{j \ge k} \sigma_j \leq  2
 \sigma_k$ we have
\begin{equation}\label{eqq1413}
\left | \frac{f(z)-f(x)}{\|z-x\|}- \langle \nabla f_k(x) \, ,
\frac{(z-x)}{\|z-x\|} \rangle \right | \leq 2\eta_k + 2(d-1)
\sigma_k /\| z- x \|
\end{equation}
if $\|z-x\| \le 2^{-N_k}$. We can now prove that
\begin{equation}\label{eqq1414}
\left \{ x\in \Rd\colon \limsup_{z\to x} \frac{|f(z)-f(x)|}{\|z-x\|}
<\infty \right \}\subseteq A.
\end{equation}
Indeed,  fix $x\notin A$, that is  $x\in \bigcap_{k\geq j} (\Rd
\setminus A_k )$ for a certain index $j$. Applying \eqref{eqq119}
in paragraph~11, for any $k\geq j$ with $j$ sufficiently large, we
have
\begin{equation*}
\| \nabla f_{k+1} (x)\|^2 \geq \|\nabla
f_k(x)\|^2+C\varepsilon_k^2
\end{equation*}
and iterating
\begin{equation*}
\| \nabla f_{k+1} (x)\|^2 \geq C \sum_{i=j}^k \varepsilon_i^2.
\end{equation*}
Since $\sum \varepsilon_k^2=\infty$, we deduce that
\begin{equation*}
\lim_{k\to\infty}\| \nabla f_{k} (x)\| =\infty.
\end{equation*}
Now choose $z$ in \eqref{eqq1413} so that $\|z-x\|=\sigma_k$ and
such that  $z-x$ is an scalar positive multiple of $\nabla f_k(x)$.
We deduce that
\begin{equation*}
\left | \frac{f(z)-f(x)}{\|z-x\|} - \|\nabla f_k(x)\|\right | \leq
2\eta_k + 2(d-1).
\end{equation*}
Consequently
\begin{equation*}
\limsup_{z\to x}  \frac{|f(z)-f(x)|}{\|z-x\|}=\infty
\end{equation*}
which proves \eqref{eqq1414}.

\medskip

\noindent $\textbf{15.}$ Finally we only have to show that $A$ has
$\sigma$-finite length. Recall that  \eqref{eqq107} tells that there
exist $\delta_k \to 0$ and a collection of cubes
$\{R_j^{(k)}\colon j=1,2,\dots\}$, $\ell (R_j^{(k)}) \le \delta_k$
such that
\begin{equation*}
\begin{split}
& A_k\subset \bigcup_j R_j^{(k)}\\
&\sum_{j\colon R_j^{(k)} \cap  \,\{\|x\|<N\}\neq \emptyset}
\ell (R_j^{(k)})\leq C(N, d)
\end{split}
\end{equation*}
for any $N$. Then $A=\bigcap_j\bigcup_{k\geq j} A_k$  has
$\sigma$-finite length. This finishes the proof of Theorem 2.

It is worth mentioning that the set $E(f)$ has $\sigma$-finite
length ( and hence Hausdorff dimension 1). It seems likely that one
could combine the construction above with one dimensional
constructions to produce a function $f \in \lambda_*(\Rd)$ such that
$E(f)$ has zero length. However there seems to be some technical
difficulties in following this plan and we have not done it.

\section{Conservative martingales}

Let $Q_0$ be the unit cube of $\R^d$. A sequence of functions $\{S_n
: n=1, 2,\ldots \}$, $S_n\colon Q_0 \to \R$, is a dyadic martingale if
for any $n=1,2,\dots$, the function $S_n$ is constant on each dyadic
cube of generation $n$ and
\begin{equation*}
\frac{1}{|Q|}\int_Q S_{n+1} dm=S_{n\,|Q}
\end{equation*}
for any dyadic cube $Q$ of generation $n$. Here $dm$ is Lebesgue
measure in $\Rd$. This corresponds to the standard notion of
martingale when the probability space is given by Lebesgue measure
in the unit cube and the filtration is the one generated by the
dyadic decomposition. A dyadic martingale $\{S_n\}$ is in the
Bloch space if there exists a constant $C=C(\{S_n \})>0$ such that
\begin{equation*}
|S_{n\,| Q} -S_{n\,| Q'} |\leq C
\end{equation*}
for any pair of dyadic cubes $Q$, $Q'$ with $\ell (Q)=\ell
(Q')=2^{-n}$ and $\overline Q\cap \overline{Q'} \neq \emptyset$. The
infimum of the constants $C>0$ satisfying the estimate above is
called the Bloch (semi)norm of $\{S_n\}$ and will be denoted by
$\|S_n\|_*$. The Little Bloch space is the subspace of those dyadic
martingales $\{S_n \}$ in the Bloch space for which
\begin{equation*}
\lim_{n\to \infty} \sup |S_{n\,| Q} -S_{n\,| Q'} | \to 0 \, ,
\end{equation*}
where the supremum is taken over all pairs of dyadic cubes $Q, Q'$
with $\ell (Q)=\ell (Q')=2^{-n}$ and $\overline Q\cap \overline{Q'}
\neq \emptyset$. It is well known that a Bloch dyadic martingale
$\{S_n\}$ may converge at no point, that is, it may happen  that
\begin{equation*}
\lim_{n\to\infty} S_n (x)
\end{equation*}
does not exist for any $x\in Q_0$. However a Bloch dyadic
martingale is bounded at a~set of maximal Hausdorff dimension,
that is, $\dim \{ x\!\in\! Q_0 \colon \limsup_{n\to\infty} |S_n
(x)|\!<\!\infty\}\!=\!d$. It is worth mentioning that the set
above may have volume zero. Similarly the set of points where a
dyadic martingale in the Little Bloch space converges may have
volume zero but it has always maximal Hausdorff dimension.
See~\cite{Ma1}. So the situation is analogous to the one described
in the introduction for the divided differences of a function in
the Zygmund class.

Let $\{e_i \colon i=1,\dots,d\}$ be the canonical basis of $\Rd$.
A sequence of mappings $\{\mathbf{S_n}\}$, $\mathbf{S_n}\colon Q_0
\to \Rd$, $n=1,2,\dots$, is called a dyadic vector-valued
martingale if for any $i=1,2,\ldots , d$, the corresponding
component $\{\langle \mathbf{S_n} \, , e_i \rangle \}_n$ is a
dyadic martingale. A dyadic vector-valued martingale
$\{\mathbf{S_n}\}$ satisfies the Bloch condition if so does each
of its components $\{ \langle \mathbf{S_n} \, , e_i \rangle \}_n$,
$i=1,\dots, d$. Similarly $\{\mathbf{S_n}\}$ is in the Little
Bloch space if for any $i =1, \ldots, d$, the scalar martingale
$\{ \langle \mathbf{S_n} ,  e_i \rangle \}$ is in the Little
Bloch~space. In contrast with the scalar case, when $d\geq 2$
there exist Bloch vectorial dyadic martingales $\{\mathbf{S_n}\}$
in $Q_0\subset \Rd$ for which
\begin{equation*}
\limsup_{n\to\infty} \| \mathbf{S_n} (x)\|=\infty,
\end{equation*}
for any $x\in Q_0$. See~\cite[p.~34]{Ma1} and \cite{U}. However
the proof of Theorem \ref{teo1} suggests that there is a natural
class of Bloch dyadic vectorial martingales which are bounded at a
set of points which has Hausdorff dimension bigger or equal to
one.

Let $d\geq 2$. A Bloch dyadic vector-valued martingale
$\{\mathbf{S_n}\}$, $\mathbf{S_n} \colon Q_0\to \Rd$, is called
conservative if there exists a constant $C=C(\{\mathbf{S_n}\})>0$
such that for any dyadic subcube $Q$ of $Q_0$ with $\ell (Q)=2^{-n}$
and any polygonal $\gamma \subset \overline{Q}$ which intersects two
different parallel faces of $Q$ one has
\begin{equation}\label{eqq51}
\left | \int_{\gamma} \mathbf{S_{n+k}} d\gamma - \int_{\gamma}
\mathbf{S_{n}} d\gamma \right |\leq C\ell (Q),
\end{equation}
for any $k=1,2,\dots$. Here $\int_{\gamma} \mathbf{S_{m}} d\gamma$
denotes the line integral of $\mathbf{S_m}$ through $\gamma$, that
is, if the curve $\gamma$ is parametrized by $\gamma \colon [0,1]\to
\Rd$, the line integral is
\begin{equation*}
 \int_{\gamma} \mathbf{S_{m}} d\gamma = \int_0^1 \langle \mathbf{S_{m}}(\gamma (t)) ,
\gamma'(t) \rangle dt.
\end{equation*}

\noindent Property \eqref{eqq51} should be understood as a one
dimensional mean value property. For instance taking $Q=Q_0$,
$n=0$ and $\gamma$ a unit segment in the direction of $e_i$,
 one deduces that for any $i=1,\dots,d$, one has
\begin{equation*}
\sup_k \left | \int_0^1 \langle \mathbf{S_k }(t e_i ) \, , e_i
\rangle dt- \langle \mathbf{S_{0 } \,| Q_0} , e_i  \rangle \right |
\leq C.
\end{equation*}
So, the mean of $ \langle \mathbf{S_k} \, , e_i \rangle$ over a unit
segment in the direction $e_i$ is, up to bounded terms, the value of
$\langle \mathbf{S_0} \, , e_i \rangle$ on the unit cube. Given a
function in the Zygmund class there is a natural conservative Bloch
dyadic vectorial martingale which governs the behavior of the
divided differences of the function. Actually let $f\colon \Rd \to
\R$ be a function in the Zygmund class. For a dyadic cube $Q\subset
Q_0\subset \Rd$ of generation $n$, the value $\mathbf{S_n}\,| Q$ of
the vectorial martingale is defined as the vector in $\Rd$ whose
components are given by
\begin{equation}\label{eqq52}
\langle \mathbf{S_n }\,|Q ,  e_{i} \rangle =\frac{1}{\ell (Q)^d}
\left ( \int_{Q^+(i)}fdA-\int_{Q^-(i)} fdA\right ), \qquad
i=1,\dots, d
\end{equation}
 where $Q^{+}(i)$ and $Q^-(i)$ are the two opposite faces of $Q$ which are orthogonal to $e_i$ so that
 the i-th coordinate of the points of $Q^{+}(i)$ is bigger than the i-th coordinate of the points  of $Q^{-}(i)$ .
 Here $dA$ is the $(d-1)$-dimensional Lebesgue measure.
 Let us now check that the martingale $\{\mathbf{S_n}\}$ is conservative. Let $Q$ be a dyadic cube
 with $\ell (Q)=2^{-n}$. Let $\gamma$ be a polygonal which intersects two different parallel faces of $Q$ at the points, say, $A$ and $B$.
  Given an integer $k>0$, choose points $A_0 = A,A_1,\dots, A_N=B$ in
 $\gamma \cap Q$ with $2^{-n-k-1} \leq \|A_{i+1}-A_i \|<2^{-n-k}$ for any $i=0,\dots, N-1$ and let $Q_i$ be the
  dyadic subcube of $Q$ which contains $A_i$ with $\ell (Q_i)=2^{-n-k}$. Lemma~\ref{lema2} tells
  that there exists an absolute constant $C>0$ such that
 \begin{equation*}
|f(B)-f(A)- \langle \mathbf{S_n}\,|Q \, , B-A \rangle | \leq
C\|f\|_*   \| B-A \|
\end{equation*}
Similarly for any $j=0,\dots, N-1$, we have
 \begin{equation*}
|f(A_{j+1})-f(A_j)- \langle \mathbf{S_{n+k}} \,|Q_j \, ,
A_{j+1}-A_j \rangle | \leq C\|f\|_*  \|A_{j+1}-A_j \|.
\end{equation*}
Then condition \eqref{eqq51} follows easily. The proof of
Theorem~\ref{teo1} applies in this context and one can obtain the
following result.

\begin{thm} \
\begin{enumerate}[(a)]

\item Let $\{\mathbf{S_n}\}$ be a conservative Bloch dyadic
vectorial martingale in $Q_0\subset \Rd$. Then the set $ \{x \in
Q_0 \colon \limsup_{n\to\infty} \|\mathbf{S_n} (x) \|<\infty \}$
has Hausdorff dimension bigger  or equal  to one.

\item Let $\{\mathbf{S_n}\}$ be a conservative dyadic vectorial
martingale in $Q_0\subset \Rd$. Assume that $\{\mathbf{S_n}\}$ is
in the Little Bloch space. Then the set
$$ \{x\in Q_0 \colon
\lim_{n\to\infty} \mathbf{S_n}(x) \textrm{ exists }  \}
$$
has Hausdorff dimension bigger or equal to $1$.
\end{enumerate}
\end{thm}

It is worth mentioning that Theorem 3 applied to the martingale
defined in  \eqref{eqq52} implies Theorem~\ref{teo1}.



\end{document}